\documentclass[11pt,oneside]{amsart}

\usepackage{amsmath,amsthm,amssymb,amsfonts,amscd}
\usepackage{setspace}
\usepackage{enumerate}
\usepackage{verbatim}

\usepackage[marginparwidth=2cm]{geometry}
\usepackage{euscript}

\usepackage{fouriernc}
\usepackage[scaled=0.875]{helvet}
\usepackage{courier}

\usepackage{epsfig}
\usepackage{graphicx}
\usepackage{wrapfig}
\usepackage{mathrsfs}
\usepackage{pinlabel}

\usepackage[babel=true]{csquotes}
\usepackage{xcolor}
\definecolor{DarkRed}{rgb}{0.9,0.,.2}
\definecolor{DarkGreen}{rgb}{0.1,0.7,0.1}
\definecolor{DarkBlue}{rgb}{0.,0.,.6}
\definecolor{grey}{rgb}{0.7,0.7,0.7}

\usepackage[colorlinks=true,urlcolor=DarkBlue,linkcolor=DarkRed,citecolor=DarkGreen,pagebackref=true]{hyperref}

\usepackage{tikz}
\usetikzlibrary{shapes,arrows}
\usetikzlibrary{fit,positioning}
\usetikzlibrary{patterns}
\usepackage{xkeyval}
\usepackage{moreverb}
\usepackage{epic}

\usepackage{marvosym}

\usepackage{todo}
\usepackage[normalem]{ulem}
\usepackage{soul}

\newcommand{\ignore}[1]{}

\theoremstyle{plain}
\newtheorem{theorem}{Theorem}
\newtheorem{proposition}[theorem]{Proposition}
\newtheorem{corollary}[theorem]{Corollary}

\newtheorem{lemma}[theorem]{Lemma}

\newtheoremstyle{theoremwithref}{}{}{\itshape}{}{\bfseries}{.}{.5em}{#1 #2 #3}
\theoremstyle{theoremwithref}

\theoremstyle{definition}

\newtheorem{remark}[theorem]{Remark}

\numberwithin{theorem}{section}
\numberwithin{equation}{section}

\newcommand{\bs}{\backslash}
\newcommand{\abs}[1]{\left|#1\right|}
\newcommand{\Abs}[1]{\left|\left|#1\right|\right|}

\newcommand{\ZZ}{\mathbb{Z}}
\newcommand{\QQ}{\mathbb{Q}}
\newcommand{\RR}{\mathbb{R}}

\newcommand{\HH}{\mathbb{H}}
\newcommand{\PP}{\mathbb{P}}
\newcommand{\RP}{\mathbb{RP}}

\newcommand{\varep}{\varepsilon}

\newcommand{\SL}{\mathrm{SL}}
\newcommand{\GL}{\mathrm{GL}}

\newcommand{\PO}{\mathrm{PO}}
\newcommand{\PSO}{\mathrm{PSO}}
\newcommand{\PSL}{\mathrm{PSL}}
\newcommand{\PGL}{\mathrm{PGL}}

\renewcommand{\so}{\mathfrak{so}}

\newcommand{\A}{\mathcal{A}}
\newcommand{\C}{\mathcal{C}}

\newcommand{\parlie}{\mathfrak{p}}

\newcommand*{\merde}[1]{\textcolor{black}{#1}}

\newcommand*{\blue}[1]{\textcolor{blue}{#1}}

\newcommand{\N}{\mathbb{N}}

\newcommand{\R}{\mathbb{R}}

\newcommand{\Hb}{\mathbb{H}}

\newcommand{\Sb}{\mathbb{S}}
\newcommand{\Pb}{\mathbb{P}}
\newcommand{\Pd}{\mathbb{P}^d}

\newcommand{\G}{\Gamma}
\newcommand{\g}{\gamma}

\renewcommand{\O}{\Omega}

\newcommand{\Cc}{\mathcal{C}}

\newcommand{\Hc}{\mathcal{H}}

\renewcommand{\Mc}{\mathcal{M}}

\renewcommand{\Hc}{\mathcal{H}}
\newcommand{\Bc}{\mathcal{B}}

\newcommand{\Pc}{\mathcal{P}}

\newcommand{\Quo}{\Omega/\Gamma}
\newcommand*{\Quotient}[2]{#1/#2}

\newcommand{\Aff}{\textrm{Aff}}
\newcommand{\Isom}{\textrm{Isom}}

\newcommand{\dO}{\partial \Omega}

\renewcommand{\leq}{\leqslant}
\renewcommand{\geq}{\geqslant}

\title[bending hyperbolic manifolds]{Properly Convex Bending of Hyperbolic Manifolds}

\author{Samuel A. Ballas}
\address{Department of Mathematics, Florida State University, Tallahassee, FL, USA}
\email{ballas@math.fsu.edu}

\author{Ludovic Marquis}
\address{Univ Rennes, CNRS, IRMAR - UMR 6625, F-35000 Rennes, France}
\email{ludovic.marquis@univ-rennes1.fr}

\begin{document}

\begin{abstract}
In this paper we show that bending a finite volume hyperbolic $d$-manifold $M$ along a totally geodesic hypersurface $\Sigma$ results in a properly convex projective structure on $M$ with finite volume. We also discuss various geometric properties of bent manifolds and algebraic properties of their fundamental groups.  We then use this result to show in each dimension $d\geqslant 3$ there are examples finite volume, but non-compact, properly convex $d$-manifolds. Furthermore, we show that the examples can be chosen to be either strictly convex or non-strictly convex. 

\end{abstract}

\maketitle
\tableofcontents

\section{Introduction}

Let $\RP^d$ denote $d$-dimensional real projective space and $\PGL_{d+1}(\R)$ denote the projective general linear group. A subset $\O$ of $\RP^d$ is called \emph{properly convex} if its closure is a convex set that is disjoint from some projective hyperplane. A properly convex set $\O$ is called \emph{strictly convex} if $\partial \O$ contains no non-trivial line segments. 

Given two properly convex domains $\O_1$ and $\O_2$ it is possible to construct a new properly convex domain $\O_1\otimes \O_2$ via an obvious product construction. A properly convex domain $\O$ is called \emph{irreducible} if the only way $\O$ can be written as such a product is if one of the factors is trivial. 

To each properly convex $\O$ we can associate an \emph{automorphism group} 
$$\PGL(\O)=\{A\in \PGL_{d+1}(\R)\mid A(\O)=\O\}$$
 and we say that $\O$ is \emph{homogeneous} if $\PGL(\O)$ acts transitively on $\O$. The Klein model, $\HH^d$, of hyperbolic space provides a quintessential example of a homogenous, irreducible, strictly convex domain, with automorphism group equal to the group, $\Isom(\HH^d)$ of hyperbolic isometries.   

 If $\G\subset \PGL(\O)$ is a discrete group then $\O/\G$ is a \emph{properly convex orbifold}. The domain $\O$ admits a $\PGL(\O)$-invariant metric, called the Hilbert metric, which gives rise to a $\PGL(\O)$-invariant measure and so it makes sense to ask if $\O/\G$ has finite volume. A domain $\O$ is called \emph{divisible} (resp.\ \emph{quasi-divisible}) if there is a discrete group $\G\subset \PGL(\O)$ such that $\O/\G$ is compact (resp.\ finite volume). In this case, the group $\G$ is said to \emph{divide} (resp.\ \emph{quasi-divide}) $\O$. The group $\Isom(\HH^d)$ is well known to contain both uniform and non-uniform lattices and as such we see that hyperbolic space is both divisible and quasi-divisible. 

In this context, there are several natural questions concerning the existence of divisible and quasi-divisible convex sets in each dimension. The main result of this paper is the following, which answer one such question in the affirmative:

\begin{theorem}\label{maintheorem}
For each $d\geq 3$ there exists an irreducible, non-homogenous, quasi-divisible $\O\subset \RP^d$. Furthermore, the domain $\O$ can be chosen to be either strictly convex or non-strictly convex.
\end{theorem}

There is a similar result, initially observed by Benoist \cite{CD1}, which shows that by combining work of Johnson--Millson \cite{JoMi} and Koszul \cite{Ko},  one can construct for all $d\geq 2$ irreducible, non-homogeneous, divisible properly convex $\O\subset \RP^d$ . However, in these examples the groups $\G$ are Gromov-hyperbolic which forces $\O$ to be strictly convex \cite{CD1}. 

The study of lattices in semisimple Lie groups provides natural context and motivation for the study of divisible and quasi-divisible convex sets. A discrete subgroup $\G$ of a semisimple Lie group $G$ is a \emph{lattice} (resp.\ \emph{uniform lattice}) if the quotient $G/\G$ has finite Haar measure (resp.\ is compact). Lattices play a role in many disparate areas of mathematics including geometric structures on manifolds, algebraic groups, and number theory, to name a few. 

Historically, much work has been dedicated to constructing and understanding lattices in semisimple Lie groups. In the late 1880's, Poincar\'e \cite{poincare} developed a technique for constructing lattices in $\SL_2(\R)$. His method is geometric and involves constructing tilings of the hyperbolic plane $\HH^2$ using isometric copies of a finite volume tile. The hyperbolic plane can be realized as the quotient of $\SL_2(\R)$ by the compact group $\textrm{SO}_2(\R)$ and the isometry group of the tiling is the desired lattice. Poincar\'e's tiling techniques were subsequently generalized by himself and others to construct concrete examples of lattices in various ``low dimensional'' Lie groups. However, explicitly constructing the required tilings in high dimensional spaces turns out to be difficult. 

It was not until 60 years later that Borel and Harish-Chandra \cite{BorelHarishChandra} developed a general technique for constructing explicit lattices in semisimple Lie groups using ``arithmetic techniques.'' Roughly speaking they showed that a semisimple Lie group $G$ could be realized as a subgroup of matrices whose entries satisfied certain integral polynomial constraints and that the subgroup $\G$ consisting of elements of $G$ with integer entries is a lattice.  In the following decade Margulis proved his seminal ``super-rigidity'' and ``arithmeticity'' results. One consequence of his work is that for most semisimple Lie groups, all of its lattices arise (up to finite index subgroups) via the previously mentioned arithmetic construction. 

As alluded to in the description of Poincar\'e's techniques, there is a strong connection between lattices and geometry. Given a Lie group $G$ we can form the associated symmetric space $G/K$, where $K$ is a maximal compact subgroup of $G$. The group $G$ acts on $X$ by isometries and if $\G$ is lattice in $G$ then $X/\G$ is a finite volume orbifold. However, because of super-rigidity, the geometry of these manifolds is typically quite rigid and does not admit deformations. 

On the other hand, the situation for properly convex manifolds (and orbifolds) is similar, but as we shall see, much more flexible. Suppose we are given a divisible (or quasi-divisible) properly convex domain $\O$ and a group $\G$ dividing (resp.\ quasi-dividing) $\O$. In this situation, we can think of $\O$ as being an analogue of the symmetric space $G/K$ and $\Gamma$ as an analogue of a lattice in $G$. In this setting there is a $\PGL(\O)$-invariant metric on $\O$ and so we can regard  $\O/\G$ as a metric object. However, despite this compelling analogy the deformation theories of lattices in semisimple Lie groups and properly convex projective manifolds have very distinct flavors. This is primarily a result of the fact that  the  group $\G$ (quasi-)dividing $\O$ is only a discrete subgroup of $\PGL_{d+1}(\R)$ and not, in general, a lattice in any Lie subgroup of $\PGL_{d+1}(\R)$. Thus $\G$ is typically not forced to satisfy super rigidity. As a result, much recent work has been focused on producing and 
understanding deformations of such manifold \cite{Goldmanthesis,Gconv,cgorb,baby_fock,marquis_moduli_surf,CLT_c,CLT_flexing,ecima,CL15,deform_sam,8knot_sam} or the survey \cite{CLM_survey}.

In fact, the proof of Theorem \ref{maintheorem} relies on a deformation theoretic argument, which we briefly outline. We start with a finite volume hyperbolic $d$-manifold $M$ that contains an embedded finite volume totally geodesic hypersurface $\Sigma$. We can realize $\HH^d$ as a strictly convex subset of $\RP^d$ and thus we can realize $M$ as $\HH^d/\G$ where $\G$ is a discrete subgroup of $\PSO(d,1)\subset \PGL_{d+1}(\R)$. Using the bending construction of Johnson and Millson \cite{JoMi} we can produce a family $\G_t\subset \PGL_{d+1}(\R)$ of subgroups such that $\G_0=\G$. We can then apply arguments of \cite{Mar} to conclude that for each $t$ the group $\G_t$ preserves a properly convex domain $\O_t$. Finally, a detailed analysis of the geometry of the cusps of $\O_t/\G_t$ allows us to conclude that $\G_t$ quasi-divides $\O_t$ and can be either strictly convex or non-strictly convex (for different choices of $M$ and $\Sigma$). 

\begin{remark}\label{r:error}
The paper \cite{Mar} by the second author contains a Theorem (Prop 6.9), a corollary of which is that the above bending construction always results in \emph{strictly convex} projective manifolds. However, the proof of this theorem contains a gap and the results of this paper show that there are non-strictly convex manifolds obtained by bending, and so Prop 6.9 of \cite{Mar} is actually false. 
\end{remark}

In the process of proving Theorem \ref{maintheorem} we are able to prove the following result.

\begin{theorem}\label{zariskidense}
The groups $\G_t$ obtained by bending $M$ along $\Sigma$ are Zariski dense for $t\neq0$. 
\end{theorem}

This result may be of independent interest because of its connection to thin groups. A group $G\subset \GL_{d+1}(\R)$ is \emph{thin} if it is Zariski dense and is also an infinite index subgroup of a lattice in $\GL_{d+1}(\R)$. Thin groups have been the object of much recent research because of their connections to number theory and a variety of Diophantine problems (see the following survey for much more detail \cite{sarnak_survey}). Theorem \ref{zariskidense} provides an infinite number of families, $\G_t$, of Zariski dense subgroups and \merde{in \cite{BallasLongThin} it is shown that for various specializations of $t$ the groups $\G_t$ give rise to thin subgroups of $\GL_{d+1}(\R)$} 

As previously mentioned, one of the steps in the proof of Theorem \ref{maintheorem} is to analyze the geometry of the ends that arise when bending a hyperbolic manifold along a totally geodesic surface. As a result of this analysis we are able to conclude that  each end of the resulting projective manifold is of one of two types which we call \emph{standard cusps} and \emph{bent cusps}, respectively (see Theorem \ref{t:holo_end_standorbend}). Bent cusps were introduced by the first author in \cite{8knot_sam} where it was shown that the complete hyperbolic structure on the figure-eight knot complement can be deformed to a properly, but not strictly, convex projective structure whose end is a bent cusp. However, these deformations of the figure-eight knot complement do not arise via the bending construction since the figure-eight knot complement contains no \emph{embedded} totally geodesic hypersurfaces.

Both standard and bent cusps are examples of \emph{generalized cusps}, introduced by Cooper--Long--Tillmann \cite{CLT15}. Loosely speaking, a generalized cusp is a properly convex projective manifold that can be foliated by nice strictly convex hypersurface that are analogous to horospheres in hyperbolic geometry. Work of the first author, D.\ Cooper, and A.\ Leitner \cite{BalCoopLeit} provides a classification of generalized cusp and their main result shows that $d$-dimensional generalized cusps fall into $d+1$ families. In this classification, standard and bent cusps form two of these families.

Given a cusp $C$ in one of these $d+1$ families it is currently an open problem to produce a properly convex manifold $M$ with non-virtually abelian fundamental group with an end that is projectively equivalent to $C$. First note that, finite volume non-compact hyperbolic manifolds give examples in each dimension of properly convex manifolds with cusp ends that are standard cusps. In dimension $2$, there are examples of properly convex manifold with cusps from each family, see \cite{annuli_2,marquis_moduli_surf}. In dimension $3$, there are examples of properly convex manifolds with cusp groups that are $\R$-diagonalizable in \cite{cd4,ecima,ballas_danciger_lee}. As previously mentioned, examples of properly convex manifolds with bent cusp ends are constructed in \cite{8knot_sam}.  Theorem \ref{examples} shows that there are examples in each dimension of properly convex manifolds with bent cusp ends. 

\merde{
During the process of revision, M. Bobb was able to construct examples of convex projective $d$-manifold with any possible non-diagonalizable type of generalized cusps \cite{bobb}, for any $d$. His construction uses multiple bending, generalizing the work of this paper. The first author was also able to construct additional examples of convex projective 3-manifold with any possible type of non-diagonalizable generalized cusps \cite{sam_gen_cusps}, without using bending. 
}

The paper is organized as follows: Section \ref{s:preli} provides necessary background material concerning properly convex geometry, introduces the paraboloid model of hyperbolic geometry, and concludes with a description of certain centralizers that are relevant throughout the paper. Section \ref{s:bending} discusses the bending construction of Johnson--Millson \cite{JoMi} at the level of representations and the level of projective structures. Section \ref{s:ends} introduces standard and bent cusps as well as discussing some of their geometric properties.  Section \ref{s:bent_ends} is dedicated to understanding what types of ends are possible for projective manifolds obtained from bending. The main results of this section are that standard and bent cusps are the only types of ends that arise when bending hyperbolic manifolds along totally geodesic hypersurfaces (Corollary \ref{c:endclassification}) and that the projective manifolds arising from bending have finite volume (Theorem \ref{t:bendingfinitevolume}). Section \ref{s:homology} describes how the topology of the pair $(M,\Sigma)$ determines the geometry of the ends of the manifolds resulting from bending. Finally, Section \ref{s:examples} is dedicated to constructing the examples needed to prove Theorem \ref{maintheorem}. 

\subsubsection*{Acknowledgements}

L. Marquis warmly thanks S. Ballas for finding the gap in the paper \cite{Mar} and his invitation to work together to fix it, which resulted in a much better theorem.
The authors would also like to thank N. Bergeron, D. Cooper, D. Long, A. Reid, for their help at different stages of the writing of this paper.

S. Ballas acknowledges the ANR Facets grant which permitted him to visit L. Marquis at the University of Rennes 1 in Summer 2015 \merde{as well as partial support from the National Science Foundation under the grant DMS 1709097.} L. Marquis was also supported by the ANR Facets grant, the ANR Finsler grant and Centre Henri Lebesgue the during the preparation of this paper.

\merde{The authors would also like to thank the anonymous referee for several helpful suggestions that helped improve the paper.}

\section{Preliminaries}\label{s:preli}
\subsection{Properly convex geometry}

Let $\RP^d$ be the space of lines through the origin in $\R^{d+1}$. More concretely, $\RP^d=(\R^{d+1}\backslash \{0\})/(x\sim \lambda x)$, where $\lambda \in \R^\times$. There is a natural projection map $P:\R^{d+1}\backslash\{0\}\to \RP^d$ taking each point to the unique line through the origin in which it is contained. This map is called \emph{projectivization}. The projectivization of a hyperplane through the origin in $\R^{d+1}$ gives rise to a \emph{hyperplane} in $\RP^d$. Given a hyperplane $H\subset \RP^d$ the set $\A=\RP^d\backslash H$ is called an \emph{affine patch} as it can be naturally identified with an affine $d$-space.

  A subset $\O$ of the real projective space $\R\Pb^d$ is said to be \emph{convex} if there exists an affine patch $\A$ of $\R\Pb^d$ such that $\O \subset \A$ and $\O$ is a convex subset of $\A$ in the usual sense. If in addition the closure $\overline{\O}$ of $\Omega$ in $\RP^{d}$ is contained in $\A$ then we say that $\O$ is \emph{properly convex}. A simple, but useful property of properly convex domains is that they do not contain complete affine lines. A properly convex open set $\O$ is \emph{strictly convex} if its boundary, $\dO$, does not contain any non-trivial line segments. 
  
To each open properly convex set $\O\subset \RP^d$ we can associate a \emph{dual} properly convex set $\O^\ast\subset \RP^{d\ast}$ as follows: consider the cone 
$$\mathcal{C}_\O^\ast=\{\phi\in \R^{(d+1)\ast}\mid \phi(x)>0,\,  \forall\  [x]\in \overline{\O}\},$$
and let $\O^\ast=P(\mathcal{C}_\O^\ast)$. If $[\phi]\in \O^\ast$ then the kernel of $\phi$ gives rise to a hyperplane disjoint from $\O$ and thus to an affine patch containing $\O$. The points $[\phi]\in \dO^\ast$ correspond to hyperplanes that intersect $\dO$, but are disjoint from $\O$. Such a hyperplane is called a \emph{supporting hyperplane to $\O$}. A point of $\dO$ is \emph{of class $\C^1$} if it is contained in a unique supporting hyperplane to $\O$. The boundary $\dO$ is then said to be \emph{of class $\C^1$} if all of its points are of class $\C^1$.  

It is easy to see that the dual of an open properly convex set is also open and properly convex. Furthermore, it is also easy to see that the notions of strict convexity and having $\C^1$ boundary are dual to one another, in the sense that if $\O$ is strictly convex (resp.\ has $\C^1$ boundary) then $\O^\ast$ has $\C^1$ boundary (resp.\ is strictly convex). 

Let $\Sb^d$ be the space of half lines through the origin in $\R^{d+1}$, which we refer to as the \emph{projective $d$-sphere}. More explicitly, $\Sb^d=(\R^{d+1}\backslash \{0\})/(x\sim \lambda x)$, where $\lambda \in \R^+$. It is easy to see that $\Sb^d$ is topologically a sphere. The group of automorphisms of $\Sb^d$ can be identified with the group $\SL^\pm_{d+1}(\R)$ of real $(d+1)\times (d+1)$ matrices with determinant equal to $\pm 1$. 

There is an obvious two-fold covering from $\pi:\Sb^d\to \RP^d$. If $H$ is a hyperplane in $\RP^d$ then the $\pi$-preimage of $H$ is double covered by an equatorial hypersphere in $\Sb^d$. Each such hypersphere partitions $\Sb^d$ into two $d$-balls  each of which is diffeomorphic (via $\pi$) to the affine patch determined by $H$. For this reason we call the complementary regions of a hypersphere \emph{affine patches}.  Given a properly convex domain $\O\subset \RP^d$ its preimage under $\pi$ consists of two components, each diffeomorphic to $\Omega$. Furthermore, the group $\PGL(\Omega)$ can be identified with a subgroup $\SL^\pm(\O)\subset \SL^\pm_{d+1}(\R)$. One convenience of the above identification is that it allows us to identify elements of $\PGL(\O)$ (which are equivalence classes of matrices) with elements of $\SL^\pm(\O)$ (which are actual matrices). We will use this identification implicitly throughout the paper. Another is that it allows us to regard $\O$ as a subset of a simply connected space.

\par{
Every properly convex open set $\O$ of $\R\Pb^d$ is equipped with a natural metric $d_{\O}$ called the \emph{Hilbert metric} defined using the cross-ratio in the following way: take any two points $x \neq y \in \O$ and draw the line between them. This line intersects the boundary $\dO$ of $\O$ in two points $p$ and $q$. We assume that $x$ is between $p$ and $y$. Then the following formula defines a metric (see Figure \ref{disttt}):
}
$$d_{\O}(x,y) =  \displaystyle \frac{1}{2}\ln \Big( [p:x:y:q] \Big)$$
\vspace*{.25em}

The topology on $\O$ induced by this metric coincides with the subspace topology coming from $\RP^d$.  The metric space $(\O,d_{\O})$ is complete, geodesic and the closed balls are compact. Furthermore, the group $\PGL(\O)$ acts properly by isometries on $\O$. 

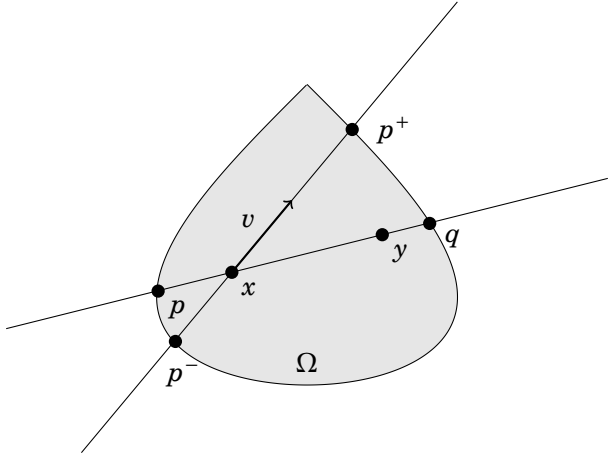
\begin{figure}[h!]
\centering
\begin{tikzpicture}
\filldraw[draw=black,fill=gray!20]
 plot[smooth,samples=200,domain=0:pi] ({4*cos(\x r)*sin(\x r)},{-4*sin(\x r)});
cycle;
\draw (-1,-2.5) node[anchor=north west] {$x$};
\fill [color=black] (-1,-2.5) circle (2.5pt);
\draw (1,-2) node[anchor=north west] {$y$};
\fill [color=black] (1,-2) circle (2.5pt);
\draw [smooth,samples=200,domain=-4:4] plot ({\x},{0.25*\x-2.25});
\draw (-1.97,-2.75) node[anchor=north west] {$p$};
\fill [color=black] (-1.98,-2.75) circle (2.5pt);
\draw (1.7,-1.8) node[anchor=north west] {$q$};
\fill [color=black] (1.63,-1.85) circle (2.5pt);
\draw [smooth,samples=200,domain=-3:2] plot ({\x},{1.2*\x-1.3});
\draw[->,distance=10pt,thick] (-1,-2.5) -- (-0.2,-1.54);
\draw (-1,-1.6) node[anchor=north west] {$v$};
\fill [color=black] (-1.75,-3.42) circle (2.5pt);
\draw (-2,-3.6) node[anchor=north west] {$p^-$};
\fill [color=black] (.6,-.6) circle (2.5pt);
\draw (.8,-.3) node[anchor=north west] {$p^+$};
\draw (0,-3.7) node {$\O$};
\end{tikzpicture}
\caption{Hilbert distance}
\label{disttt}
\end{figure}

\par{
The Hilbert metric gives rise to a Finsler structure on $\O$ defined by a very simple formula. Let $x$ be a point of $\O$ and $v$ a vector of the tangent space $T_x \O$ of $\O$ at $x$. The quantity  $\left. \frac{d}{dt}\right| _{t=0} d_{\O}(x,x+tv)$ defines a Finsler structure $F_{\O}(x,v)$ on $\O$. Moreover, if we choose an affine chart $\A$ containing $\O$ and a euclidean norm $|\cdot|$ on $\A$, we get:
}
\begin{equation}\label{e:finsler}
F_{\O}(x,v) = \left. \frac{d}{dt}\right| _{t=0} d_{\O}(x,x+tv) = \frac{|v|}{2}\Bigg(\frac{1}{|xp^-|} + \frac{1}{| xp^+|} \Bigg)
\end{equation}

\par{
Where $p^-$ and $p^+$ are the intersection points of the line through $x$ spanned by $v$ with $\dO$ and $|ab|$ is the distance between points $a,b$ of $\A$ for the euclidean norm $|\cdot|$ (see Figure \ref{disttt}). The regularity of this Finsler metric is determined by the regularity of the boundary $\dO$ of $\O$, and the Finsler structure gives rise to a Hausdorff measure $\mu_{\O}$ on $\O$ which is absolutely continuous with respect to Lebesgue measure, called the \emph{Busemann volume}.
}
\par{
More concretely, if $A \subset \O$ is a Borel subset, then the Busemann volume of $A$, denoted $\mu_{\O}(A)$, is computed as 
$$\int_{A}\frac{\alpha_d}{\mu_L(B_z^{\O}(1))}d\mu_L(z),$$
where $\mu_L$ is the Lebesgue measure on $(\A,| \cdot |)$, $\alpha_d$ is the Lebesgue volume of a unit $d$-ball, and $B_z^\O(1)$ is the unit ball for the Hilbert norm on the tangent space $T_z \O$. It is easy to see that the measure defined by this formula does not depend on the choice of the affine patch containing $\Omega$ or on the euclidean norm $| \cdot |$ on $\A$ since $\mu_{\O}$ is a Hausdorff measure of $(\O,d_{\O})$. Furthermore, if $\Gamma$ is a discrete subgroup of $\PGL(\O)$ we see that $\mu_{\O}$ is $\Gamma$-invariant and thus descends to a measure $\mu_{\O/\G}$ on $\O/\G$.   
}
\\
\par{
We close this section by mentioning some useful ``contravariance'' properties of the Hilbert metric and Busemann volume of different domains.

\begin{proposition}\label{p:compa_hilbert}
Let $\O_1\subset \O_2$ be two properly convex open sets, and let $x,y\in \O_1$. Then $d_{\O_2}(x,y)\leq d_{\O_1}(x,y)$.
\end{proposition}

\begin{proof}
The proposition is a consequence of the following inequality whose verification is a straightforward computation. If $a,x,y,b\in \RP^1$ and $t>0$ then
$$[a:x:y:b+t]\leq [a:x:y:b]$$
\end{proof}
 
\begin{proposition}(see Colbois-Verovic-Vernicos \cite[Proposition 5]{aire_cvv})\label{p:compa_busemann}
Let $\O_1 \subset \O_2$ be two properly convex open sets; then for any Borel set $D$ of $\O_1$, we have $\mu_{\O_2}(D) \leqslant \mu_{\O_1}(D)$.
\end{proposition}
}


\subsection{The paraboloid model of $\HH^d$}\label{s:paraboloid}

In this section we discuss a projective model of hyperbolic space that can be viewed as a projective analogue of the upper half space model. Specifically, there is a distinguished point, $\infty$, in the boundary of this model and automorphisms fixing $\infty$ have a particularly nice form. 

Let $Q_d$ be the quadratic form on $\RR^{d+1}$ given by 
\begin{equation}\label{quadform}
 x_2^2+\ldots x_d^2-2x_1x_{d+1}
\end{equation}
It is easily verified that $Q_d$ has signature $(d,1)$ and so the projectivization of its negative cone gives a projective model of $\HH^d$ with isometry group $\PO(Q_d)$. More explicitly, if we let $\{e_i\}_{i=1}^{d+1}$ be the standard basis for $\RR^{d+1}$ and $\{e_i^\ast\}_{i=1}^{d+1}$ the corresponding dual basis, then we see that the negative cone of $Q_d$ is disjoint from the hyperplane dual to $e_{d+1}^\ast$ and so we can realize this model for $\HH^d$ as a paraboloid whose homogeneous coordinates are 
\begin{equation}\label{e:homogenouscoords}
 \{[x_1:\ldots:x_d:1]\mid x_1>(x_2^2+\ldots+x_d^2)/2\}
\end{equation}
Furthermore, the boundary of $\HH^d$ can be identified with the space of isotropic lines for the form $Q_d$. Again, we can explicitly realize $\partial \HH^d$ in homogeneous coordinates as
\begin{equation}\label{e:boundarycoords}
 \{[x_1:\ldots:x_{d}:1]\mid x_1=(x_1^2+\ldots+x_{d}^2)/2\}\cup\{[1:0:\ldots:0]\}
\end{equation}
We henceforth use these identifications implicitly and will refer to the point $[1:0:\ldots:0]\in \partial \HH^d$ as $\infty$. 

Let $\so(Q_d)$ be the Lie algebra of $\PSO(Q_d)$ and let $\parlie_d$ be the Lie algebra of the group $P_d$ of parabolic translations fixing $\infty$. This Lie algebra can be described explicitly as 
$$\parlie_d=\left\{\begin{pmatrix}
                  0 & u_1 & \ldots &u_{d-1} &0\\
                  0 & 0 &\ldots & 0 & u_1\\
                  \vdots & \vdots& \ddots &\vdots & \vdots\\
                  0 & 0 & \ldots & 0 & u_{d-1}\\
                  0 & 0 & \ldots & 0 &0
                 \end{pmatrix} \mid (u_1,\ldots, u_{d-1})\in \RR^{d-1}\right\}$$
As a Lie algebra, $\parlie_d$ is isomorphic to $\RR^{d-1}$ and the exponential map provides an isomorphism between $\parlie_d$ and $P_d$. We will often write elements of $P_d$ in the following block form

\begin{equation}\label{pnblockform}
 \begin{pmatrix}
  1 & v^t & \frac{\abs{v}^2}{2}\\
  0 & I & v\\
  0 & 0 & 1
 \end{pmatrix}
\end{equation}
where $v$ is a (column) vector in $\RR^{d-1}$, $I$ is the $(d-1)\times(d-1)$ identity matrix, and the zeros represent zero matrices of the appropriate and shapes. If $g\in P_d$ then the vector $v$ in \eqref{pnblockform} is called the \emph{translation vector} of $g$.

\merde{
There is a foliation $\mathcal{F}$ of $\RP^d\bs[\ker e_{d+1}^\ast]$ that is (leafwise) invariant under $P_d$. In terms of \eqref{e:boundarycoords} each leaf of $\mathcal{F}$ is of the form
\begin{equation}\label{foliation}	
\mathcal{F}_c=\{[x_1:\ldots,:x_d:1]\mid x_1=(x_2^2+\ldots+x_d^2)/2+c\}
\end{equation}
 for some $c\in \R$. When $c>0$ this corresponds to the foliation of $\HH^d$ by horospheres centered at $\infty$ (see Section \ref{s:generalizedcusps} for more details). Furthermore, the orbit closures of the action of  $P_d$ in $\RP^d$ consist of $[e_1]$, $[\ker e_{d+1}^\ast]$, and closures of leaves of the above foliation.
}

Let $H$ be a hyperplane in $\HH^d$. All such hyperplanes are \merde{in the same $\PSO(Q_d)$ orbit} and so after applying an element of $\PSO(Q_d)$ we can assume that $H$ is given by the intersection of $\HH^d$ and the projective hyperplane defined by the equation $x_2=0$. We will refer to this hyperbolic hyperplane as $\HH^{d-1}_0$. \merde{While the choice of the plane $x_2=0$ may initially seem odd, it provides a convenient way to projectively embed the paraboloid model of $\HH^{d-1}$ into the paraboloid model of $\HH^d$.} Let $\PSO(Q_d;d-1,1)$ be the index two subgroup of the stabilizer in $\PSO(Q_d)$ of $\HH^{d-1}_0$ that preserves both components of the complement of $\HH^{d-1}_0$ in $\HH^d$.  The subgroup of parabolic translations of $\PSO(Q_d;d-1,1)$, which we denote by $P_{d-1}^0$, can be identified with the image under the exponential map of the subalgebra $\parlie_{d-1}^0$ of $\parlie_d$ of elements whose translation vector has zero as its first component.

\subsection{Centralizers}
In order to define bending and later to understand the geometry of the ends of manifolds arising from bending it will be necessary to describe the centralizers in $\PGL_{d+1}(\RR)$ of several of the groups described in the previous section.  

The identity component of the centralizer of $\PSO(Q_d; d-1,1)$ in $\PSO(Q_d)$ is trivial, however when regarded as a subgroup of $\PGL_{d+1}(\RR)$ it has 1-dimensional centralizer which is described in the following lemma (similar lemmas appear in \cite{JoMi}, \cite[Lem 3.2.3]{BallasThesis} and \cite[Lem 3.3]{Mar})

\begin{lemma}\label{psocentralizer}
The identity component $C_{d-1}$ of the centralizer of $\PSO(Q_d; d-1,1)$ in $\PGL_{d+1}(\RR)$ is one dimensional and is equal to the one parameter group with infinitesimal generator 
\begin{equation}\label{e:c_t}
C=\begin{pmatrix}
   -1 \\
    & d \\
    & & -1 \\  
    &&& \ddots \\
   &&&& -1
  \end{pmatrix}
\end{equation}
Specifically, $C_{d-1}=\{c_t\mid t\in \RR\}$, where $c_t=\exp(t C)$. 
\end{lemma}

The next lemma describes the centralizer of $P_d$ in $\PGL_{d+1}(\RR)$.

\begin{lemma}\label{pncentralizer}
The centralizer $\mathcal{Z}(P_d)$ of $P_d$ in $\PGL_{d+1}(\RR)$ consists of matrices of the following block form
\begin{equation}\label{e:pncentralizer}
 \begin{pmatrix}
   1 & u^t &b\\
   0 & I & u\\
   0 & 0 & 1
  \end{pmatrix}
\end{equation}
where $u\in\RR^{d-1}$ and $b\in \RR$. 
\end{lemma}
\begin{proof}
 \merde{
A simple computation using block matrices shows that any element of the form \eqref{e:pncentralizer} commutes with every element of $P_d$. Next, Let $G$ be the subgroup of elements in $\PGL_{d+1}(\R)$ of the form
 $$\begin{pmatrix}
 	1 & 0 & b\\
 	0 & I & 0\\
 	0 & 0 & 1
 \end{pmatrix}.$$
 Observe that every element of the form $\eqref{e:pncentralizer}$ can be written as a product of an element of $P_d$ and an element of $G$. Furthermore, the group $G$ acts transitively on the leaves of $\mathcal{F}$ (see \eqref{foliation}). }
 \merde{
 Next, suppose that $B\in \mathcal{Z}(P_d)$. The point $[e_1]$ (resp.\ $[e_{d+1}^\ast]$) is the unique point in $\RP^d$ (resp.\ $\RP^{d\ast}$) preserved by $P_d$. As $B$ commutes with all elements of $P_d$, the group $B$ must also fix $[e_1]$ and $[e_{d+1}^\ast]$ and permute the leaves of $\mathcal{F}$.}
 
  \merde{Let $\mathcal{F}_c$ be a leaf of this foliation.  There is an element $C$ in $G$ so that $CB$ preserves $\mathcal{F}_c$. The closure of $\mathcal{F}_c$ bounds a copy of $\HH^d$ in $\RP^d$ that contains the point $\infty$ in its boundary and both $P_d$ and $CB$ preserves this copy of $\HH^d$. It follows that $CB$ is a hyperbolic isometry that commutes with every element of $P_d$. Hence $CB\in P_d$ and the result follows. }

\end{proof}

We conclude this subsection by identifying the centralizer of $P_{d-1}^0$ in $\PGL_{d+1}(\RR)$. The group $P_{d-1}^0$  \merde{acts trivially on} a unique line, $C$, in $\RP^d$ and a unique line, $C^\ast$, in $\RP^{d\ast}$ (a line in $\RP^{d\ast}$ corresponds to a pencil of hyperplanes). Namely $C$ is the line spanned by $[e_1]$ and $[e_2]$ and $C^\ast$ is pencil of hyperplanes corresponding to the line in $\RP^{d \ast}$ spanned by $[e_2^\ast]$ and $[e_{d+1}^\ast]$. \merde{The point $[e_1]$ is the unique point of $C$ contained in the core of the pencil defined by $C^\ast$ and $[e_{d+1}^\ast]$ is the only point in $C^\ast$ whose kernel contains the line $C$. It follows that both $[e_1]$ and $[e_{d+1}^\ast]$ are both preserved by any element that centralizes $P^0_{d-1}$}. Furthermore, any point in $\RP^d$ (resp.\ hyperplane  in $\RP^{d\ast}$) that is invariant under $P_{d-1}^0$ is contained in this line (resp.\ pencil). Consequently, any element of $\PGL_{d+1}(\RR)$ that centralizes $P_{d-1}^0$ must also preserve this line (resp.\ pencil). See Figure \ref{fig:coordinates}.

\merde{The group $P^0_{d-1}$ also preserves the foliation $\mathcal{F}$ leafwise. Furthermore, each leaf $\mathcal{F}_c$ admits a foliation whose leaves are 
$$\mathcal{F}_{c,d}=\{[x_1:x_2:\ldots,:x_d:1]\in \mathcal{F}_c\mid x_2=d\},$$
where $d\in \R$. This foliation of $\mathcal{F}_c$ is preserved leafwise by $P^0_{d-1}$. 
}

 We repeatedly use these facts in the proof of the following lemma:

\begin{figure}
\centering
\definecolor{ffxfqq}{rgb}{1,0.5,0}
\definecolor{qqqqff}{rgb}{0,0,1}
\definecolor{ffqqqq}{rgb}{1,0,0}
\begin{tikzpicture}[line cap=round,line join=round,>=triangle 45,x=0.8cm,y=0.8cm,scale=0.6]
\clip(-7.85,-7.93) rectangle (9.97,9);
\draw [line width=1.6pt] (0,0) circle (4.8cm);
\draw [rotate around={-53.91:(2.81,2.05)},line width=1.2pt,color=ffqqqq,fill=ffqqqq,fill opacity=0.08] (2.81,2.05) ellipse (3.89cm and 0.51cm);
\draw [rotate around={89.96:(0,0)},line width=1.2pt,color=ffqqqq,fill=ffqqqq,fill opacity=0.19] (0,0) ellipse (4.8cm and 0.58cm);
\draw [rotate around={-74.92:(1.49,0.4)},line width=1.2pt,color=ffqqqq,fill=ffqqqq,fill opacity=0.07] (1.49,0.4) ellipse (4.64cm and 0.51cm);
\draw [rotate around={-111.18:(-2,0.77)},line width=1.2pt,color=ffqqqq,fill=ffqqqq,fill opacity=0.06] (-2,0.77) ellipse (4.48cm and 0.46cm);
\draw [rotate around={47.49:(-2.96,2.72)},line width=1.2pt,color=ffqqqq,fill=ffqqqq,fill opacity=0.07] (-2.96,2.72) ellipse (3.55cm and 0.31cm);
\draw [line width=1.2pt] (-5,7)-- (-7,5);
\draw [line width=1.2pt] (-7,5)-- (7,5);
\draw [line width=1.2pt] (7,5)-- (9,7);
\draw [line width=1.2pt,color=qqqqff] (-6,6)-- (8,6);
\draw [line width=1.6pt,color=ffxfqq] (-1,5)-- (2,8);
\begin{scriptsize}
\fill [color=black] (0,6) circle (1.5pt);
\draw[color=black] (-0.4,6.4) node {$[e_1]=\infty$};
\fill [color=black] (0,-6) circle (1.5pt);
\draw[color=black] (0,-6.4) node {$[e_{d+1}]$};
\draw (7,6.4) node {$\rightarrow \quad [e_2]$};
\draw (1.5,8.01) node {$\nearrow$};
\draw (2,8.5) node {$[e_3],\dots, [e_d]$};
\draw (6,5.41) node {$T_{[e_1]} \partial \O = \Pb(\{ e_{d+1}^\ast = 0 \})$};
\draw (-5,6.4) node {$C$};
\draw (0,0) node {$\Hb_0^{d-1}$};
\end{scriptsize}
\end{tikzpicture}
\caption{\label{fig:coordinates}This picture illustrates our choice of coordinates. Some cross sections of the pencil $C^\ast$ with $\Hb^d$ are colored in red.}
\end{figure}
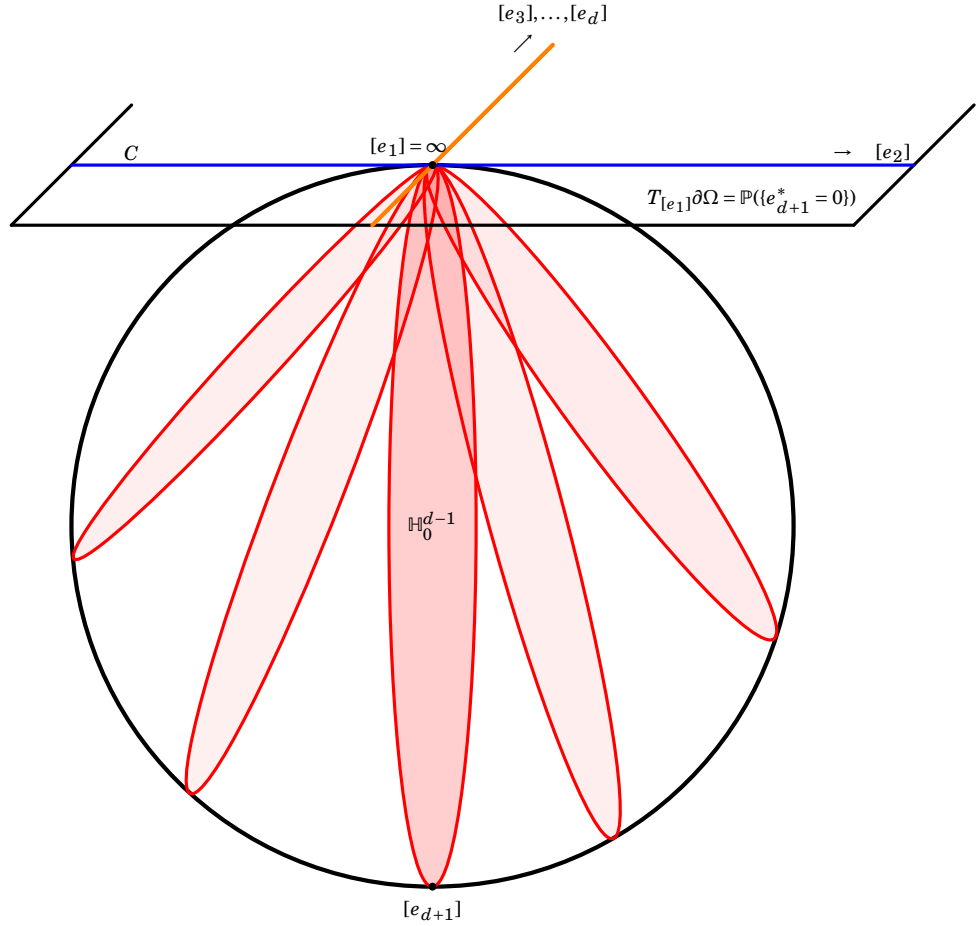

\begin{lemma}\label{pn-1centralizer}
 The centralizer $\mathcal{Z}(P_{d-1}^0)$ of $P_{d-1}^0$ in $\PGL_{d+1}(\RR)$ consists of elements with block form 
 \begin{equation}\label{e:pn-1centralizer}
 \begin{pmatrix}
    1 & a & u^t & z\\
    0 & b & 0 & c\\
    0 & 0 & I & u\\
    0 & 0 & 0 & 1
   \end{pmatrix}
 \end{equation}
where $a,c,z\in \RR$, $b\in \RR^\times$, $u\in \RR^{d-2}$, and $I$ is the $(d-2)\times (d-2)$ identity matrix.  
\end{lemma}

\begin{proof}
 Again, elements of the form \eqref{pn-1centralizer} form a Lie group and a simple computation using block matrices shows that any element of the form \eqref{e:pn-1centralizer} commutes with every element of $P_{d-1}^0$. 
 
 \merde{Next, observe that $\mathcal{Z}(P_d)$ is a subgroup of the set of matrices of the form \eqref{e:pn-1centralizer}. Let $B\in \mathcal{Z}(P^0_{d-1})$ Pick a leaf $\mathcal{F}_{c,0}$ of the foliation above. Then there is $C\in \mathcal{Z}(P_d)$ so that $CB$ preserves $\mathcal{F}_{c,0}$. The closure $\overline{\mathcal{F}_{c,0}}\subset [\ker e_2^\ast]$ can be identified with $\partial \HH^{d-1}\subset \RP^{d-1}$, arguing as in the proof of Lemma \ref{pncentralizer} and observing that $CB$ preserves both $[e_1]$ and $[e_{d+1}^\ast]$, it follows that 
 $$CB=\begin{pmatrix}
1 & a & u^t& z\\
0 & b & d^t & c\\
0 & e & I & u\\
0 & 0 & 0 & 1	
\end{pmatrix},
$$}

\merde{where $d$ and $e$ are (column) vectors in $\R^{d-2}$. Finally, since $CB$ centralizes $P_{d-1}^0$ it also preserves both $C$ and $C^\ast$ and so both $d$ and $e$ must be zero, and so $CB$ is of the form \eqref{e:pn-1centralizer}. The result then follows by observing that elements of the form \eqref{pn-1centralizer} form a group of which $C$ is an element. }
\end{proof}

\section{Bending}\label{s:bending}

\par{
Let $M$ be an \merde{orientable} finite volume hyperbolic $d$-manifold and $\Sigma$ an embedded finite-volume totally geodesic hypersurface. We denote the fundamental groups of $M$ and $\Sigma$ by $\Gamma$ and $\Delta$, respectively. In this section we will show how to construct a family of properly convex projective structure on $M$ by ``bending'' along $\Sigma$. More information about bending and its relationship to projective structure can be found in \cite{JoMi} and \cite{Mar}. By Mostow rigidity there is a unique (up to isometry) hyperbolic structure on $M$ and so we get a discrete and faithful representation $\rho_0:\Gamma\to \PSO(Q_d)$ (unique up to conjugacy \merde{in ${\rm PO}(Q_d)$}) from the holonomy of this structure. We will henceforth use this structure to identify $\widetilde{M}$ with $\HH^d$ and $\Gamma$ with a subgroup of $\PSO(Q_d)$. Furthermore, by assuming that we have choosen a base point $\tilde{x} \in \HH^d$ whose projection to $M$ is contained in $\Sigma$ and that the lift of $\Sigma$ containing $\tilde{x}$ is $\HH^{d-1}_0$ we may assume that $\Delta$ is 
a subgroup of $\PSO(Q_d; d-1,1)$.  
}

\subsection{Bending at the level of representations}

\par{
We first describe the bending construction at the level of representations. The construction depends on whether or not the hypersurface $\Sigma$ is separating. }
\\
\par{
If $\Sigma$ is separating then $M\backslash \Sigma$ has two components $M_1$ and $M_2$ with fundamental groups $\Gamma_1$ and $\Gamma_2$. Furthermore, we can decompose $\Gamma$ as the amalgamated free product 
}
\begin{equation}\label{amalprod}
 \Gamma = \Gamma_1\ast_{\Delta}\Gamma_2
\end{equation}

\par{
The representation $\rho_0$ gives rise to two representations $\rho_0^i:\Gamma_i\to \PSO(Q_d)$ given by restricting $\rho_0$ to $\Gamma_i$ for $i=1,2$. We define two families of representations of $\Gamma_1$ and $\Gamma_2$, respectively, as follows. Let $\rho_t^1=\rho_0^1$ and let $\rho_t^2=c_t\rho_0^2c_t^{-1}$, where $c_t$ is the element defined in \eqref{e:c_t}. Since $c_t$ belong to $C_{d-1}$ the identity component of the centralizer $\mathcal{Z}(\PSO(Q_d; d-1,1))$ of $\PSO(Q_d; d-1,1)$, these two families of representations agree on $\Delta$ and thus give a family of representations $\rho_t:\Gamma\to \PGL_{d+1}(\RR)$. 
}
\\
\par{
If $\Sigma$ is non-separating then $M\backslash \Sigma$ has a single component $M_\Sigma$ with fundamental group $\Gamma_\Sigma$ and we can write $\Gamma$ as the following HNN extension:
}

\begin{equation}\label{HNN}
 \Gamma = \Gamma_\Sigma\ast_s
\end{equation}
\par{
where $s$ is the stable letter. We can define a family of representations $\rho_t:\Gamma\to \PGL_{d+1}(\RR)$ as follows.  We define $\rho_t$ to be equal to $\rho_0$ when restricted to $\Gamma_\Sigma$ and equal to $c_t \rho_0(s)$ when restricted to the stable letter. Since $c_t$ centralizes $\rho_0(\Delta)$ this gives a well defined family of representations $\rho_t:\Gamma\to \PGL_{d+1}(\RR)$.
}

\subsection{Bending at the level of projective structures}

\par{
In this section we show, these two families of deformations defined by bending are both holonomies of projective structures on $M$ arising from a certain type of projective deformation. Let $\widetilde \Sigma$ be the union of all the lifts of $\Sigma$ to $\HH^d$. Recall that the hyperplane $\HH^{d-1}_{0}$ is one such lift 
}

\par{
We begin with the case where $\Sigma$ separates $M$ into $M_1$ and $M_2$. For $i\in \{1,2\}$ let $N_i=M_i\cup \Sigma$.  Let $\widetilde{N_i}$ be the copy of the respective universal cover of $N_i$ in $\HH^d$ that contains $\HH^{d-1}_0$ in its boundary.  Combinatorially, $\widetilde M$ can be described 
$$\widetilde M=(\Gamma\times \widetilde N_1)/\Gamma_1\sqcup (\Gamma\times \widetilde N_2)/\Gamma_2,$$

\noindent  where $\alpha\in\Gamma_i$ acts on $\Gamma\times \widetilde N_i$ by $\alpha\cdot(\gamma,p)=(\gamma\alpha^{-1},\alpha\cdot p).$ Additionally, if $p\in \widetilde N_1\cap \widetilde N_2=\HH^{d-1}_0$ then we identify the point $(\gamma,p)\in \Gamma\times \widetilde N_1$ with the point $(\gamma,p)\in \widetilde N_2$. The action of $\Gamma$ on $\widetilde M$ is given by 

\begin{equation}\label{e:amalaction}
 \gamma\cdot[(\gamma',p)]=[(\gamma\gamma',p)] {\rm\ for\ } \gamma\in \Gamma {\rm\ and\ } [(\gamma',p)]\in \widetilde M
\end{equation}
With this description of the universal cover, the developing map is easy to describe. Let $D_0:\HH^d\to \RP^d$ be the developing map for the complete hyperbolic structure on $M$ and let $c_t\in \PGL_{d+1}(\RR)$ be the element from \eqref{e:c_t}. Define a new developing map $D_t:\HH^d\to \RP^d$ by  

\begin{equation}\label{e:amalbenddev}
D_t([(\gamma,p)])=\left\{\begin{matrix}
                  \rho_t(\gamma)D_0(p) & \rm{ if } & p\in \widetilde N_1 \\
                  \rho_t(\gamma)c_tD_0(p) & {\rm if } & p\in \widetilde N_2
                 \end{matrix}\right.
\end{equation}

\noindent It is a simple exercise to verify that $D_t$ is well defined and $\rho_t$-equivariant. 
}

\par{
The case where $\Sigma$ is non-separating can be treated similarly. Let $N=\overline{M_\Sigma}$ and observe that there are two components of the universal cover of $N$ in $\HH^d$ that contain $\HH^{d-1}_0$ and we can order these lifts so that $\rho_0(s)$ takes the first lift to the second lift. With this convention we let $\widetilde N$ be the first of the two lifts. The universal cover of $M$ can again be described combinatorially as 
$$\widetilde M=(\Gamma\times \widetilde N)/\Gamma_\Sigma,$$

\noindent where $\alpha\in \Gamma_\Sigma$ acts by $\alpha\cdot(\gamma,p)=(\gamma\alpha^{-1},\alpha\cdot p)$. The action of $\Gamma$ on $\widetilde M$ is given by 
\begin{equation}\label{e:HNNaction}
 \gamma\cdot[(\gamma',p)]=[(\gamma\gamma',p)] {\rm\ for\ } \gamma\in \Gamma {\rm\ and\ } [(\gamma',p)]\in \widetilde M.
\end{equation}

\noindent The new developing map $D_t:\HH^d\to \RP^d$ is given by
\begin{equation}\label{e:HNNbenddev}
 D_t([(\gamma,p)])=\rho_t(\gamma)D_0(p).
\end{equation}

\noindent It is again easily verified that $D_t$ is well defined and $\rho_t$-equivariant. 
}

\par{
As a result, we have constructed a family of projective structures with developing/holonomy pair $\Mc_t = (D_t,\rho_t)$ which we call \emph{bending of $M$ along $\Sigma$}. By work of \cite[Lem.\ 5.4 and Lem.\ 5.5]{JoMi} it is known that for $t\neq 0$ these projective structures are not hyperbolic, but thanks to the following theorem it is known that they remain properly convex.

}

\begin{theorem}\label{t:prop_conv}\cite[Theorem 3.7]{Mar}
Let $(\Mc_t)_{t\in \R}$ be the bending of $M$ along $\Sigma$. The projective structure $\Mc_t$ on $M$ is properly convex.
\end{theorem}

\section{Geometry of the ends}\label{s:ends}

\par{
In this section we give a detailed description of the ends of the manifolds obtained by bending. The section begins by describing the geometry of two different types of ends. We then proceed to show that (up to passing to a finite sheeted cover) these are the only two types of ends that can arise in manifolds obtained by bending. The main component of this is Theorem \ref{t:holo_end_standorbend}.  
}

\subsection{Standard and bent cusps}\label{s:generalizedcusps}

In this section we describe in detail the geometry of two different types of ends. It should be noted that these types of ends are specific instances of \emph{generalized cusps}, which were introduced by Cooper--Long--Tillmann in \cite{CLT15}.

\subsubsection*{Standard cusps}

We begin by letting $\Lambda$ be a lattice in the $(d-1)$-dimensional Lie group $P_d$. Let $\A$ be the affine patch corresponding to $[e_{d+1}^\ast]$, $\A$ is diffeomorphic to $\RR^{d}\cong \RR\times \RR^{d-1}$ with affine coordinate $(x,v)$, where $x\in \RR$ and $v\in \RR^{d-1}$. \merde{The first factor in this decomposition is called the \emph{vertical direction} note that in all affine figures (e.g.\ Figure \ref{fig:domain_3d}) that the vertical direction is vertical.} For $c\in \RR$ we can define the function $f_c:\RR^{d-1}\to \RR$ by $v\mapsto \frac12 \abs{v}^2+c$.

In these coordinates the paraboloid model of $\HH^d$ can be realized as the epigraph of $f_0$. Furthermore, each hyperbolic horosphere (resp.\ horoball) centered at $\infty$ is given by the graph (resp.\ epigraph) of $f_c$ for some $c>0$. These horospheres give us a foliation of $\HH^d$ by convex hypersurfaces. This foliation is preserved leafwise by the action of $\Lambda$ (each leaf is the $P_d$ orbit of some point).  \merde{Varying the vertical coordinate} in this product structure gives another foliation of $\A$ by lines passing through $\infty$. The group $\Lambda$ also preserves this foliation.

These two foliations are transverse to one another and the space of these lines can be identified with the second factor of the product structure. The action of $\Lambda$ on the space of lines is by euclidean translations. Projection onto the second factor also endows each of the horospheres with a euclidean structure. Thus $\HH^d/\Lambda \cong T^{d-1}\times(0,\infty)$ and the torus fibers $T^{d-1}$ are euclidean. This is nothing but a projective version of a familiar construction from hyperbolic geometry. We call a manifold of the form $\HH^d/\Lambda$ a \emph{standard torus cusp} and a manifold of the form $\HH^d/\Lambda'$, where $\Lambda'$ contains $\Lambda$ as a finite index normal subgroup, a \emph{standard cusp}.

\subsubsection*{Bent cusps}\label{s:bentcusps}

Next, let $\Lambda$ be a lattice in the $(d-1)$-dimensional Lie group $B_d\subset \PGL_{d+1}(\RR)$ consisting of elements of the form 
\begin{equation}\label{e:Bd}\begin{pmatrix}
     1 & 0 & v^t & \frac{\abs{v}^2}{2}-b\\
    0 & e^b& 0 & 0\\
    0 & 0 & I & v\\
    0 & 0 & 0 & 1
  \end{pmatrix}
\end{equation}
where $b\in \RR$ and $v\in \RR^{d-2}$. The group $B_d$ preserves $\A$, which we now realize as $\RR\times \RR\times \RR^{d-2}$ with affine coordinated $(x,y,v)$, where $x,y\in \RR$ and $v\in \RR^{d-2}$. \merde{Again, the first coordinate is called the \emph{vertical direction}.} Let $c\in \RR$ and define $g_c:\RR^+\times \RR^{d-2}\to \RR$ by $(y,v)\mapsto \frac12 \abs{v}^2-\log(y)+c$. Let $\Bc^d$ be the epigraph of $g_0$. The graphs of $g_c$ for $c>0$ give a foliation of $\Bc^d$ by strictly convex hypersurfaces. The Hessian of $g_c$ is positive definite at each point in its domain and so we get that $\Bc^d$ is convex. It is not hard to see that $\Bc^d$ is properly, but not strictly convex. In particular, the domain $\Bc^d$ contains a unique segment in its boundary, which in these coordinates is the segment $[e_1,e_2]$. We henceforth refer to $[e_1]$ as $p_\infty^+$, $[e_2]$ as $p_\infty^-$, and $[e_1,e_2]$ as $s_\infty$. 

We call the graphs (resp.\ epigraphs) of the $g_c$ \emph{horospheres centered at $s_\infty$} (resp.\ \emph{horoballs centered at $s_\infty$}). Again, the leaves of this foliation are $B_d$ orbits and thus this foliation is preserved leafwise by $\Lambda$. The lines coming from the first factor of the product structure are concurrent to $p_\infty^+$ and give a foliation of $\Bc^d$ which is preserved by $\Lambda$ and this foliation by lines is again transverse to the foliation by horospheres.

\begin{center}
 \begin{figure}
  \includegraphics[scale=.6]{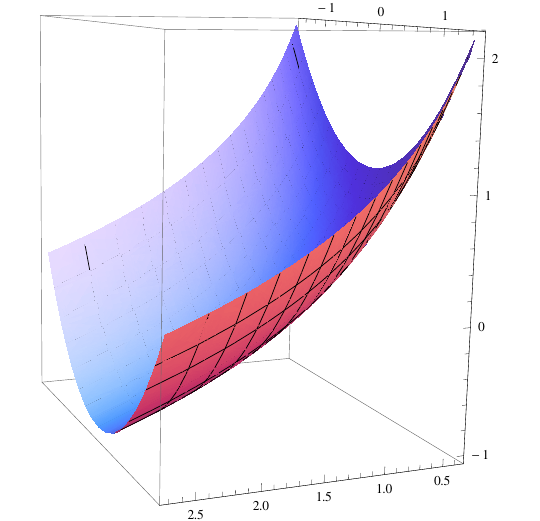}
  \caption{The domain $\Bc^3$}
  \label{fig:domain_3d}
 \end{figure}
\end{center}

The space of lines can be identified with a subset of the product of the second and third factors, the action of $\Lambda$ on the space of lines is by affine transformations, but is no longer by euclidean isometries. More precisely, the action on the third factor is by euclidean translations and the action on the second factor is by homothety. Projection to the space of lines endows the horospheres with an affine structure. The quotient $\Bc^d/\Lambda$ is still diffeomorphic to $T^{d-1}\times (0,\infty)$, but now the torus sections $T^{d-1}$ are affine, but no longer euclidean. We call a manifold of the form $\Bc^d/\Lambda$ a \emph{bent torus cusp} and a manifold of the form $\Bc^d/\Lambda'$, where $\Lambda'$ contains $\Lambda$ as a finite index normal subgroup, a \emph{bent cusp}.  
Next, we discuss some interesting Lie subgroups of $B_d$ as well as their orbits. First, let $H_{di}$ be 1-dimensional subgroup of $B_d$ consisting of elements such that $v=0$ (see \eqref{e:Bd}). We refer to $H_{di}$ as the \emph{group of pure dilations} and to its non-trivial elements as \emph{pure dilations}. Let $\gamma$ be a pure dilation such that $b< 0$ (see \eqref{e:Bd}), then $p_\infty^-$ is a repulsive fixed point of $\gamma$ and $p_\infty^+$ is an attractive fixed point of $\gamma$. If $x\in \partial\Bc^d\backslash s_\infty$, then the curve $(\g^t(x))_{t\in \R}\cup s_\infty$ is the boundary of a two dimensional convex subset, $\omega_x$, of $\Bc^d$, see Figure \ref{f:omega_x}. 

Next, let $H_{tr}$ be the $(d-2)$-dimensional subgroup of $B_d$ consisting of elements such that $b=0$. We refer to $H_{tr}$ as the \emph{group of pure translations} and to its non-trivial elements as \emph{pure translations}. The group of pure translations acts trivially on $s_\infty$. Furthermore, for any point $x\in \partial \Bc^d\backslash s_\infty$, the $H_{tr}\cdot x\cup p_\infty^+$ is the boundary of a totally geodesic copy of $\HH^{d-1}$ in $\Bc^d$.

To summarize, we see that every cross section of $\Bc^d$ \merde{with a 2-plane containing $s_\infty$} is of the form $\omega_x$ for some $x\in \Bc^d$. Furthermore, every cross section of $\Bc^d$ with an hyperplane that contains $p_\infty^+$ and transverse to $s_\infty$ is a $(d-1)$-dimensional ellipsoid (provided the cross section is non-empty).

\begin{center}
 \begin{figure}
  \includegraphics[scale=.3]{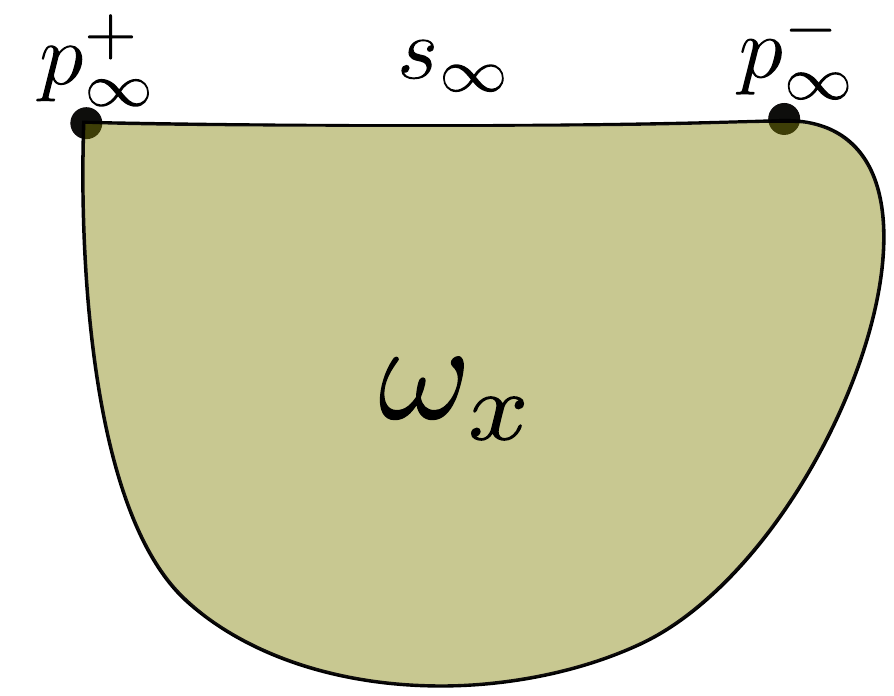}
  \caption{The domain $\omega_x$.} 
    \label{f:omega_x}
 \end{figure}
\end{center}

\subsection{Volumes of cusp neighborhoods}
In this section we show that the cusp neighborhoods defined in the previous section have finite Busemann volume. The precise statement is as follows:

\begin{theorem}\label{t:fvcuspnbhd}
Let $\O$ be either $\HH^d$ or $\Bc^d$, let $G$ be either $P_d$ or $B_d$ and let $\Hc \subset \O$ be a horoball (i.e. the convex hull of an orbit of $G$). If $\Lambda \subset G$ is a lattice then $\Lambda$ preserves $\Hc$ and $\Hc/\Lambda$ is a properly convex submanifold of $\O/\Lambda$. If $d\geqslant 3$ then $\Hc/\Lambda$ is a finite volume submanifold of $\O/\Lambda$. 
\end{theorem}

\begin{proof}
With the exception of the claim about $\Hc/\Lambda$ having finite volume when $d\geq 3$ the rest of the theorem follows from the discussion in the previous subsection. Furthermore, when $\O=\HH^d$ the Hilbert metric on $\O$ is equal to the hyperbolic metric and so in this case the Busemann volume coincides with the hyperbolic volume. In this case the fact that $\Hc/\Lambda$ is finite volume in $\O/\Lambda$ is a well known fact from hyperbolic geometry that follows from a simple computation.
 
Assume  now that $\O=\Bc^d$ and, \merde{using the coordinates from Section \ref{s:bentcusps}}, view $\Bc^d \subset \A\cong \RR\times \RR\times \RR^{d-2}$ with coordinates $(x,y,v)$ and recall that $\O$ is the epigraph of a function whose domain is $\RR^+\times \RR^{d-2}$. \merde{Recall that $B_{z_0}^{\O}(1)$ is the unit ball for Hilbert norm centered at the origin in $T_{z_0}(\O)\cong \A$}. The proof of this case is similar to \cite[Prop 3]{8knot_sam} and proceeds by showing that when $z_0=(x_0,y_0,v_0)$ with $x_0$ large that $B_{z_0}^{\O}(1)$ contains a simplex of Lebesgue volume comparable to $x_0^{d/2}$. Let $\mathcal{D}$ be a fundamental domain for the action of $\Lambda$ on $\O$. We can assume that $\mathcal{D}$ is the intersection of $\O$ with the cone over a compact set $C\subset\RR^+\times \RR^{d-2}$ with cone point $p_\infty^+$. The compact set $C$ can be taken to be a fundamental domain for the affine action of $\Lambda$ on $\RR^+\times \RR^{d-2}$.  
 
Let $z_0\in \mathcal{D}$ and \merde{recall that since $\O\subset \A$ that we can} identify $T_{z_0} \O$ with $\RR\times \RR\times \RR^{d-2}$. \merde{Specifically, we use} the coordinates $(x,y,v)$ from Section \ref{s:bentcusps}, \merde{and we put $z_0$ at the origin of the tangent space $T_{z_0} \O$.} Consider the \merde{vector} $w_1=(x_0,0,0)\in T_{z_0} \O$. A simple computation using \eqref{e:finsler} shows that 
 \begin{equation}\label{e:length1}
  \Abs{w_1}=\frac{x_0}{2x_0-\abs{v_0}^2+2\log y_0}.
 \end{equation}
Since $(y_0,v_0)\in C$ which is compact, we see from \eqref{e:length1} that $\Abs{w_1}<1$ for sufficiently large $x_0$ and so in this case $w_1\in B_{z_0}^{\O}(1)$. 

Next, let $w_2=(0,\varep,0)$, where $\varep>0$. Another simple computation shows that 
\begin{equation}\label{e:length2}
 \Abs{w_2}=\frac{\varep}{2\left(y_0-\exp \left(\frac{\abs{v_0}^2}{2}-x_0\right) \right)}.
\end{equation}
Since $(y_0,v_0)$ is confined to a compact set in $\RR^+\times \RR^{d-2}$ we see that for sufficiently small $\varep$ and sufficiently large $x_0$ that $w_2\in B_{z_0}^{\O}(1)$.

Next, \merde{extend $\{v_0\}$ to a basis for $\R^{d-2}$ and} perform the Gram-Schmidt process \merde{(with respect to the standard inner product on $\R^{d-2}$)} to obtain an orthonormal basis $\{v'_0,\ldots v'_{d-3}\}$ of $\RR^{d-2}$ and let $w_i=\left(0,0,\sqrt{x_0}v_{i-3}'\right)$ for $3\leqslant i\leqslant d$. Another computation \merde{using \eqref{e:finsler}} shows that
\begin{equation}\label{e:length3}
 \Abs{w_3}=\frac{\sqrt{2(x^2_0+x_0\log y_0})}{2(x_0+\log y_0)-\abs{v_0}^2}
\end{equation}
and 
\begin{equation}\label{e:length4}
 \Abs{w_i}
 =\frac{\sqrt{x_0}}{\sqrt{2(x_0+\log y_0-\frac{1}{2}\abs{v_0}^2)}}
\end{equation}
for $4\leqslant i\leqslant d$. Again, since $(y_0,v_0)$ is constrained to a compact set, we see that for large values of $x_0$ that $w_i\in B_{z_0}^{\O}(1)$ for $3\leq i\leq d$.

We now see that for sufficiently large $x_0$ that $\{(0,0,0),w_1,\ldots, w_d \}\subset B_{z_0}^{\O}(1)$. Let $S$ be the simplex formed by taking the convex hull of this set. Since $B_{z_0}^{\O}(1)$ is the unit ball of a norm it is convex and thus contains the simplex, $S$. The Lebesgue measure of $S$ is easily computed as $C_{d,\varep} x_0^{d/2}$, where  $C_{d,\varep}$ is a constant depending only on $d$ and $\varep$

As a result we see that there is a compact set $K\subset \mathcal{D}$ such that for $z_0\in \mathcal{D}\backslash K$ there is a simplex in $B_{z_0}^{\O}(1)$ of volume at least $C_{d,\varep} x_0^{d/2}$. Therefore
$$\mu_{\O/\Lambda}(\Hc/\Lambda)=\mu_{\O}(\mathcal{D})=\int_{K}\frac{\alpha_d}{\mu_L(B_{z}^{\O}(1))}d\mu_L(z)+\int_{\mathcal{D}\backslash K}\frac{\alpha_d}{\mu_L(B_{z}^{\O}(1))}d\mu_L(z)$$
$$\leqslant\int_{K}\frac{\alpha_d}{\mu_L(B_{z}^{\O}(1))}d\mu_L(z)+\int_{\mathcal{D}\backslash K}\frac{\alpha_d}{C_{d,\varep} x^{d/2}}d\mu_L(z)<\infty$$

\end{proof}

\begin{remark}
 If $d=2$ and $\O=\HH^2$ then $\Hc/\Lambda$ is a finite area submanifold of $\HH^2/\Lambda$. Conversely, if $d=2$ and $\O=\Bc^2$ then $\Hc/\Lambda$ is an infinite area submanifold of $\Bc^2/\Lambda$ (See \cite{Surf_ludo}). 
\end{remark}

\section{Classification of the ends}\label{s:bent_ends}
This section is dedicated to understanding the ends of manifolds that arise by bending. Specifically, we show that the ends of a properly convex manifold obtained from bending a finite volume hyperbolic manifold along a finite volume totally geodesic hypersurface are finitely covered by either a standard torus cusp or a bent torus cusp (Theorem \ref{t:holo_end_standorbend}). We close this section by showing that the manifolds obtained by bending will always have finite Busemann volume (Theorem \ref{t:bendingfinitevolume}). Recall that $\mathcal{M}_t=(\O_t,\G_t)$ is the family of properly convex projective structures obtained by bending $M$ along $\Sigma$. 

\subsection{Classification of the ends}
\par{
The goal of this subsection is to show that each end of a manifold obtained by bending is (up to passing to a finite sheeted cover) either a standard torus cusp or a bent torus cusp. We begin by describing the topology of the ends of $M$ as well as their intersection with the totally geodesic hypersurface $\Sigma$. 
}
\par{
We recall that a manifold without boundary $M$ is \emph{topologically tame} when it is the interior of a compact manifold $\overline{M}$. In that case, the union $\Pc$ of all the conjugates of the fundamental groups of the connected components of the boundary of $\overline{M}$ is called \emph{the family of the peripheral subgroups of $M$}.
}
\par{
It is well known that finite volume hyperbolic manifolds are topologically tame. We let $\{T_i\}_{i=1}^k$ denote the boundary components of $\overline{M}$, which we refer to as \emph{cusp cross sections}. Each of them is a flat $(d-1)$-manifold, i.e a manifold that admit a metric with constant sectional curvature equal to zero, see the first paragraph of \ref{s:generalizedcusps}.
}

\par{
Let $T$ be one such cusp cross section and let $\Gamma_\infty$ be a fixed \emph{peripheral subgroup for $T$}, i.e.\ a fixed representative of the conjugacy class of $\pi_1(T)$ in $\Gamma=\pi_1(M)$. After conjugating by an element of $\PSO(Q_d)$ we can assume that $\Gamma_\infty$ fixes $\infty\in \partial \Hb^d$.  
}
\par{
Since $\Sigma$ is also a finite volume hyperbolic manifold it is also tame and has a finite set $\{D_i\}_{i=1}^l$, of cusp cross sections which are $(d-2)$-dimensional flat manifolds. Suppose one of the cusp cross sections of $\Sigma$ intersects $T$. Without loss of generality assume that it is $D_1$ and let $\Delta_\infty$ be a fixed peripheral subgroup for $D_1$. By choosing $\Delta_\infty$ appropriately we can assume that $\Delta_\infty\subset \Gamma_\infty$. 
}
\par{
It is possible for another cusp cross section, say $D_2$, of $\Sigma$ to intersect $T$. Since $\Sigma$ is embedded in $M$ we see that $D_1$ and $D_2$ are \emph{parallel} in the sense that the universal covers of $D_1$ and $D_2$ are parallel hyperplanes in the universal cover of $T$ which is $\RR^{d-1}$ with the usual euclidean structure. Thus we see that $D_1$ and $D_2$ are freely homotopic in $M$ and thus have fundamental groups which are conjugate in $\Gamma$ (but not in $\Delta = \pi_1(\Sigma)$).
}

\par{
In order to understand the structure of the ends we first show that (up to conjugacy) the group $\rho_t(\Gamma_\infty)$ is highly constrained. Specifically, we show that $\rho_t(\Gamma_\infty)$ is virtually a lattice in one of the two $(d-1)$-dimensional abelian Lie groups $P_d$ or $B_d$.
}
\par{
Specifically, if we let $\Gamma_\infty^{Tr}$ be the finite index subgroup consisting of parabolic translations of $\Gamma_\infty$ we show that a conjugate of $\rho_t(\Gamma_\infty^{Tr})$ is contained in one of the aforementioned abelian Lie groups.
}

In the proof of Theorem \ref{t:holo_end_standorbend} we encounter two additional Lie groups 

\begin{equation}\label{e:P'}
P'_d = \left\{
\begin{pmatrix}
   1 & a & u^t & \frac{-a^2+|u|^2}{2}\\
    & 1 & 0 & -a\\
    &  & I & u\\
    &  &  & 1
  \end{pmatrix}
  :
a\in \RR,\, u\in \RR^{d-2}
\right\}
\end{equation}

and 

\begin{equation}\label{e:B'}
B'_d = \left\{
\begin{pmatrix}
   1 & 0 & u^t & \frac{|u|^2}{2}+t\\
    & e^t & 0 & 0\\
    &  & I & u\\
    &  &  & 1
  \end{pmatrix}
  :
t\in \RR,\, u\in \RR^{d-2}
\right\}
\end{equation}
Note that both $P'_d$ and $B'_d$ contain $P^0_{d-1}$ as a codimension 1 Lie subgroup. 

\begin{remark}
We stress a difference between $P_d$ and $P'_d$. We recall that $P_d$ preserves the quadratic form $Q_d$ of signature $(d,1)$ defined in \ref{quadform}. Furthermore, a simple computation shows that $P'_d$ also preserves a quadratic form $Q'_d$ defined on $\RR^{d+1}$, of signature $(d-1,2)$ given by: 
\begin{equation}\label{quadformprime}
 -x_2^2+x_3^2+\ldots x_d^2-2x_1x_{d+1}
\end{equation}
Recall that $\A$ is the affine patch corresponding to $[e_{d+1}^*]$. If we look first at the orbit of a point $x$ in $\A$ under $P'_d$ in the inhomogeneous coordinates obtained by setting $x_{d+1}=1$, we get that the orbit of $p=(p_1,\cdots,p_d)\in \A$ is the $(d-1)$-quadric hypersurface 
$$
S =\{\, x=(x_1,\cdots,x_d) \in \A \,\,|\,\, -x_2^2+x_3^2+\ldots x_d^2-2x_1 = Q'_d(p) \, \}
$$
This quadric hypersurface $S$ is a hyperbolic paraboloid and hence its convex hull in $\A$ is all of $\A$.
\end{remark}

Using the following lemma we can rule out the possibility that $\rho_t(\Gamma_\infty^{Tr})$ is a lattice in either of these Lie groups by showing that neither of $P'_d$ nor $B'_d$ contains a lattice that preserves a convex domain. 

\begin{lemma}\label{l:prime_not_conv}
Let $\Lambda$ be a lattice in $P'_d$ or $B'_d$. If $\Omega$ is an open convex set \merde{preserved} by $\Lambda$ then $\O$ contains an affine line. Consequently, such a lattice does not preserve a properly convex open subset of $\R\Pd$.
\end{lemma}

\begin{proof}
Suppose first that $\Lambda$ is a lattice of $P'_d$. Since $\O$ is open it must contain a point $p\in \A$. For the present time we will regard $\O$ as a subset of $\Sb^d$. From \eqref{e:P'} we see that each $\gamma\in P'_d$ is determined by a pair $(a,u)\in \R\times \R^{d-2}$, and we denote the corresponding element $\gamma_{(a,u)}$. Since $\Lambda$ is a lattice we can find a sequence $(\alpha_n:=\alpha_{(a_n,u_n)})_{n\in \N}$ such that the sequence $(a_n)_{n\in \N}$ is bounded and $(\abs{u_n})_{n\in \N}$ diverges to $\infty$. A simple computation shows that $(\alpha_n\cdot p)_{n\in\N}$ converges to $[e_1]$, and so $[e_1]\in\dO$. On the other hand, we can also find a sequence of elements $(\beta_n:=\beta_{(a_n,u_n)})_{n\in \N}$ in $\Lambda$ such that $(\abs{u_n})_{n\in \N}$ is bounded and $(\abs{a_n})_{n\in \N}$ diverges to $\infty$. Again, it is easy to see that $(\beta_n\cdot p)_{n\in \N}$ converges to $[-e_1]$, and so $[-e_1]\in \dO$. By convexity, we see that $\O$ must contain an entire affine line connecting $[e_1]$ and $[-e_1]$. This contradicts the fact that $\O$ is properly convex.

Finally suppose that $\Lambda$ is a lattice of $B'_d$. Again it is better to work in the projective sphere $\Sb^d$. Since $\O$ is open it contains a point $p\in \A\smallsetminus [\ker e_2^\ast]$. From \eqref{e:B'} we see that each $\gamma\in B'_d$ is determined by a pair $(t,u)\in \R\times \R^{d-2}$, and we denote the corresponding element $\gamma_{(t,u)}$. Since $\Lambda$ is a lattice we can find a sequence $(\alpha_n:=\alpha_{(t_n,u_n)})_{n\in \N}$ such that the sequence $(t_n)_{n\in \N}$ is bounded and $(\abs{u_n})_{n\in \N}$ diverges to $\infty$. A simple computation shows that $(\alpha_n\cdot p)_{n\in\N}$ converges to $[e_1]$, and so $[e_1]\in\dO$. On the other hand, we can also find a sequence of elements $(\beta_n:=\beta_{(t_n,u_n)})_{n\in \N}$ in $\Lambda$ such that $(\abs{u_n})_{n\in \N}$ is bounded and $(t_n)_{n\in \N}$ diverges to $-\infty$. Again, it is easy to see that $(\beta_n\cdot p)_{n\in \N}$ converges to $[-e_1]$, and so $[-e_1]\in \dO$. By convexity, we see that $\O$ must contain an entire affine line connecting $[e_1]$ and $[-e_1]$. Again this contradicts proper convexity of $\O$. 
\end{proof}


\begin{theorem}\label{t:holo_end_standorbend}
Let $(\Mc_t)_{t\in \R}$ be the bending of $M$ along $\Sigma$. Let $\G_\infty$ be a peripheral subgroup of $\G$. The holonomy $\rho_t(\G_\infty)$ is virtually a lattice in a conjugate of $P_d$ or $B_d$.
\end{theorem}

\begin{proof}
\par{
Let $T$ be a cusp cross section of $M$. We begin by analysing the following simple case. Suppose that no cusp cross section of $\Sigma$ intersects $T$ then $\G_\infty$ is contained in the fundamental group of a component of $M\smallsetminus \Sigma$ thus by construction $\rho_t(\Gamma_\infty)=\rho_0(\Gamma_\infty)$, and so $\rho_t(\Gamma_\infty^{Tr})$ is a lattice in $P_d$. 
}

\par{
Next, suppose that the cusp cross section of $\Sigma$ intersects $T$. Let $\Delta_\infty$ and $\Gamma_\infty$ be as before and let $\Delta_\infty^{Tr}$ be the subgroup of parabolic translations in $\Delta_\infty$. By construction of $\rho_t$, the group $\rho_t(\Delta_\infty^{Tr})$ is a lattice of $P_{d-1}^0$. Furthermore, the quotient $\Gamma^{Tr}_\infty/\Delta_\infty^{Tr}\cong \ZZ$. Let $\gamma$ be any element of $\Gamma_\infty^{Tr}$ that projects to a generator, $\overline{\gamma}$, in this cyclic quotient. The group $P_{d-1}^0$ \merde{preserves each hyperplane of} a unique pencil of hyperplanes $C^\ast$. Namely it preserves \merde{leafwise} the pencil of hyperplanes corresponding to the line in $\RP^{d \ast}$ spanned by $e_2^\ast$ and $e_{d+1}^\ast$, and in fact $P_{d-1}^0$ acts trivially on this pencil. Since $\Gamma_\infty^{Tr}$ is abelian we get that $\rho_t(\gamma)$ also preserves $C^\ast$.
}
\\
\par{
 The next lemma describes how the abelian Lie group in which $\rho_t(\Gamma_\infty^{Tr})$ is contained depends only on the dynamics of $\rho_t(\gamma)$ on $C^\ast$ and thus concludes the proof.
}
\end{proof}

\begin{lemma}\label{l:latticeliegroup}
The action of $\rho_t(\gamma)$ on $C^\ast$ is orientation preserving and either parabolic or hyperbolic. Furthermore, if the action of $\rho_t(\gamma)$ is parabolic then $\rho_t(\Gamma_\infty^{Tr})$ is conjugate to a lattice in $P_d$ and if $\rho_t(\gamma)$ is hyperbolic then $\rho_t(\Gamma_\infty^{Tr})$ is conjugate to a lattice in $B_d$
\end{lemma}

\begin{proof}
The matrix $\rho_t(\gamma)$ commutes with every element of $\Delta_\infty^{Tr}$ and thus centralizes $P_{d-1}^0$. Thus by Lemma \ref{pn-1centralizer} we see that 
\begin{equation}\label{e:rhogammaform}
 \rho_t(\gamma)=\begin{pmatrix}
                   1 & \alpha & v^t & z\\
   0 & \beta & 0 & \delta\\
   0 & 0 & I & v\\
   0 & 0 & 0 & 1
                 \end{pmatrix}
\end{equation}


We first show that the action of $\rho_t(\gamma)$ on $C^\ast$ is non-trivial and orientation preserving. The action of $\rho_t(\gamma)$ on the universal cover $\O_t$ of $\Quotient{\O_t}{\G_t}$ send every lift of $\Sigma$ to a different lift of $\Sigma$, each lift of $\Sigma$ gives a point of $C^\ast$, hence the action on $C^\ast$ is non-trivial. Moreover, from the action of $\rho_t(\gamma)$ on the universal cover $\O_t$, we see that the action of $\rho_t(\gamma)$ on $C^\ast$ is topologically conjugated to an increasing homeomorphism, thus the action of $\rho_t(\gamma)$ on $C^\ast$ is orientation preserving. Furthermore, the action of $\rho_t(\gamma)$ on $C^\ast$ fixes $[e_{d+1}^\ast]$, and is thus not elliptic. It remains to prove that if $\rho_t(\gamma)$ is parabolic (resp.\ hyperbolic) then $\rho_t(\G_\infty^{Tr})$ is conjugate into $P_d$ (resp.\ $B_d$).

The action of $\rho_t(\gamma)$ on $C^\ast$ is given (in appropriate projective coordinates) by 
$$\begin{pmatrix}
   \beta & \delta\\
   0 & 1
  \end{pmatrix}
$$


Since the action of $\rho_t(\gamma)$ is orientation preserving we get that $\beta>0$. Henceforth we will write $\beta=e^b$ and we see that the action of $\rho_t(\gamma)$ is parabolic if and only if $b=0$.

Next, we assume that $b=0$ and prove that $\rho_t(\gamma)$ can be conjugated into $P_d$ by an element that centralizes $P_{d-1}^0$. By assumption we have
$$\rho_t(\gamma)=\begin{pmatrix}
                  1 & \alpha & v^t & z\\
                  0 & 1 & 0 & \delta\\
                  0 & 0 & I & v\\
                  0 & 0 & 0 & 1
                 \end{pmatrix}
$$

Since the actions of $\rho_t(\gamma)$ on both $C^\ast$ and on the unique $P_{d-1}^0$-invariant line $C$ of $\R\Pb^d$ are non-trivial, we get that neither $\alpha$ or $\delta$ can be zero. Furthermore, by conjugating by an element of the form
\begin{equation}\label{e:central}
\begin{pmatrix}
   1 & 0 & 0 & 0\\
   0 & e & 0 & f\\
   0 & 0 & I & 0\\
   0 & 0 & 0 & 1
  \end{pmatrix}
\end{equation}
we can assume that $\alpha=\pm\delta$ and that $z=(\abs{v}^2\pm\alpha^2)/2$. Note that the element in \eqref{e:central} centralizes $P_{d-1}^0$ by Lemma \ref{pn-1centralizer}

Thus this case will be complete if we can show that $\alpha=\delta$. Suppose for contradiction that $\alpha=-\delta$, then we see that $\rho_t(\gamma)\in P'_d$ and thus $\rho_t(\Gamma_\infty^{Tr})$ is a lattice in $P_d'$. Thus by Lemma \ref{l:prime_not_conv} we get that $\rho_t(\Gamma)$ cannot preserve an open properly convex set, which contradicts Theorem \ref{t:prop_conv}. We conclude that $\rho_t(\gamma)\in P_d$ and hence that $\rho_t(\Gamma_\infty^{Tr})$ is a lattice in $P_d$.

Assume now that the action of $\rho_t(\gamma)$ on $C^\ast$ is hyperbolic. We complete the proof by showing that $\rho_t(\gamma)$ is conjugate into $B_d$ by an element normalizing $P_{d-1}^0$. Since the action of $\rho_t(\gamma)$ is hyperbolic we can assume that 
$$\rho_t(\gamma)=\begin{pmatrix}
                  1 & \alpha & v^t & z\\
                  0 & e^b & 0 & \delta\\
                  0 & 0 & I & v\\
                  0 & 0 & 0 & 1
                 \end{pmatrix}
$$
such that $b\neq0$. By replacing $\gamma$ with its inverse we can assume without loss of generality that $b>0$. Furthermore, by conjugating by an element in the normalizer of $P_{d-1}^0$ of the form 
$$\begin{pmatrix}
   e^2 & f & 0 & 0\\
   0 & 1 & 0 & g\\
   0 & 0 & e\cdot I & 0\\
   0 & 0 & 0 & 1
  \end{pmatrix}
$$
we can assume that $\alpha=\delta=0$ and that $z=1/2\abs{v}^2\pm b$. The case where $z=1/2\abs{v}^2+b$ cannot occur, since if it did, $\rho_t(\Gamma_\infty^{Tr})$ would be conjugate to a lattice in $B'_d$. This gives rise to a contradiction similar to that of the parabolic case, thanks to Lemma \ref{l:prime_not_conv} and Theorem \ref{t:prop_conv}.  
\end{proof}

\begin{remark}\label{r:S1affstruct}
As we have seen $\rho_t(\Gamma_\infty^{Tr})$ preserves $C^\ast\cong \RP^1$ and this representation descends to give an action of the cyclic group $\Gamma_\infty^{Tr}/\Delta_\infty^{Tr}$ on $C^\ast$. We denote by $\omega_t^\ast$ the convex open subset of $C^*$ consisting of the hyperplanes of $C^\ast$ that intersect $\O_t$. The set $\omega_t^\ast$ is a domain of discontinuity for the action of the cyclic group $\Gamma_\infty^{Tr}/\Delta_\infty^{Tr}$ on $C^\ast$. If we identify $\R$ with $\RP^1\bs\{\infty\}$, where $\infty$ is a fixed point of $\rho_t(\gamma)$ in $C^\ast$ then we can projectively identify $\omega_t^\ast$ with a subset of $\R$ and $\rho_t(\gamma)$ with a element of the affine group $\Aff(\R)$. Hence we get an affine structure on $S^1$.

As a consequence of Lemma \ref{l:latticeliegroup} we see that the holonomy of this affine structure is either parabolic or hyperbolic, depending on how $\rho_t(\gamma)$ acts on $C^\ast$. In this way we can associate an affine structure on $S^1$ to each cusp of $M$ and we see that whether or not this affine structure is euclidean determines whether or not the cusp is standard. 
\end{remark}

\subsection{Horoballs in manifolds arising from bending}

In this section we discuss some existence and configuration results that will be used to prove that manifolds obtained by bending have finite volume.  
For $t\neq 0$ the domains $\partial\Omega_t$ will not have strong regularity properties. For example, their boundaries are never $\Cc^2$. However, the following lemma shows that these domains can be approximated by the horoballs introduced in section \ref{s:generalizedcusps}, which are smooth almost everywhere.

\begin{lemma}\label{l:horoballapprox}
 Let $M$ be a finite volume hyperbolic manifold and let $\Sigma$ be a finite volume totally geodesic hypersurface. Let $\Mc_t=\Omega_t/\Gamma_t$ be a projective manifold obtained by bending $M$ along $\Sigma$. Let $\G_p$ be a peripheral subgroup of $\G_t$. Then there exist horoballs $\Hc_{int}$ and $\Hc_{ext}$ centered at  a face $s_p\subset \partial \O_t$ such that:
\begin{enumerate}
\item $\Hc_{int}$ and $\Hc_{ext}$ are $\Gamma_p$-invariant.
\item $\Hc_{int} \subset \Omega_t \subset \Hc_{ext}$
\end{enumerate}

\end{lemma}

\begin{proof}
\par{
By Theorem \ref{t:holo_end_standorbend}, we know that $\G_p$ contains a finite index normal subgroup $\G_p'$ that is conjugate to a lattice in either $P_d$ or $B_d$, and we will henceforth assume that we have conjugated $\G_p'$ into either $P_d$ or $B_d$. The horoballs $\Hc_{int}$ and $\Hc_{ext}$ that we construct will be epigraphs of the functions $f_c$ and $g_c$ that we defined in Section \ref{s:generalizedcusps}. Thus $\Hc_{int}$ and $\Hc_{ext}$ will easily seen to be invariant under $\G_p'$ and hence $(1)$ is satisfied since $\G_p'$ is a discrete normal subgroup \merde{of $\G_p$}.

Let us first treat the case where $\G_p'$ is a lattice in $P_d$. In this case $\G_p$ has a unique fixed point $s_p$ and a unique invariant hyperplane $p^\ast$ that contains $s_p$. The point $s_p$  (resp. $p^\ast$) is an accumulation point of $\G_p'$-orbit of any point in $\O_t$ (resp. $\O_t^\ast$) and so $s_p\in \partial \O_t$ and $p^\ast\in\partial \O_t^\ast$. Thus $p^\ast$ corresponds to a supporting hyperplane to $\O_t$ at $s_p$. 
}
\par{ 
Let $\A$ be the affine patch defined by $p^\ast$. In these coordinates the points of $\partial \Omega_t$ that are not contained in the kernel of $p^\ast$ or in any segment \merde{included in $\partial \O_t$} through $p$ can be realized as the graph of $h_t:U_t\subset \RR^{d-1}\to \RR$, where $U_t$ is a open convex $\G_p'$-invariant subset of $\R^{d-1}$ and $h_t$ is a continuous convex function  (Here we are identifying $\RR^{d-1}$ with the space of lines through $s_p$ that are not contained in the kernel of $p^\ast$). 

It is easy to see that the only open convex $\G_p'$-invariant subset of $\R^{d-1}$ is $\R^{d-1}$ and so $U_t=\R^{d-1}$. If we let $f_0$ be the function defined in Section \ref{s:generalizedcusps} then in order to find $\Hc_{int}$ satisfying $(2)$ we need to find a positive constant $D$ such that $h_t<f_0+D$.

Let $K\subset \RR^{d-1}$ be a compact fundamental domain for the affine action of $\G_p'$ on $\RR^{d-1}$ and choose $D$ so that $h_t\vert_{K} < f_0 \vert_{K} +D$. Suppose for contradiction that there is a point $u\in U_t$ such that $h_t(u) \geqslant f_0(u)+D$. By continuity of $h_t$ we can find $v\in U_t$ such that $h_t(v)=f_0(v)+D$. Furthermore, we can find $\gamma\in \G_p'$ such that $\gamma v\in K$. As a result we get that $\gamma\cdot (h_t(v),v)=\gamma\cdot(f_0(v)+D,v)$. By equivariance properties of $h_t$ and $f_0$ we get that $(h_t(\gamma v),\gamma v)=(f_0(\gamma v)+D,\gamma v)$, but this contradicts our choice of $D$. The existence of $\Hc_{ext}$ follows from a similar argument where we find a positive constant $E$ such that $f_0-E<h_t$. This completes the proof of $(2)$ in this case. 

In the case where $\G_p'$ is a lattice in $B_d$ the group $\G_p'$ now has 2 distinct fixed points $p_+$ and $p_-$. Each of these points is an accumulation point of the $\G_p'$-orbit of a point in $\O_t$ and so both $p_+$ and $p_-$ are contained in $\partial \O_t$. A similar argument shows that the group $\G_p'$ has two fixed points $p_{\pm}^\ast\in \partial \O_t^\ast$. One of these dual fixed points, say $p_+^\ast$, corresponds to a supporting hyperplane for $\O_t$ and we let $s_p$ be the segment connecting $p_+$ and $p_-$ that is contained in $\dO_t$. 

Again we see that in the affine patch corresponding to $p_+^\ast$ the points of $\partial \Omega_t$ that are not contained in the kernel of $p^\ast$ or in any segment containing $p$  can be realized as the graph of $h_t:U_t\subset \RR^{d-1}\to \RR$, where $U_t$ is a open convex $\G_p'$-invariant subset of $\R^{d-1}$ and $h_t$ is a continuous convex function. Similar to the previous case we see that the only open convex $\G_p'$-invariant subsets of $\R^{d-1}$ are $\R^\pm\times \R^{d-2}$, and so without loss of generality, we can asssume that $U_t=\R^+\times \R^{d-2}$. 

If we let $g_0$ be the function defined in section \ref{s:generalizedcusps} then we can again find positive constants $D$ and $E$ such that $f_0-E<h_t<f_0+D$, and thus we can find horoballs $\Hc_{int}$ and $\Hc_{ext}$ satisfying $(2)$ and $(3)$. 
}

\end{proof}

Let $H\subset \G$ be a subgroup, and let $X\subset \O$ be a subset, then we say that $X$ is \emph{$(\G,H)$-precisely invariant} or just \emph{precisely invariant} if the groups are clear from context whenever 

\begin{itemize}
\item $X$ is invariant under $H$
\item If $\gamma\in \G$ and $\gamma\cdot X\cap X\neq \varnothing$ then $\gamma\in H$. 
\end{itemize}

Precisely invariant subsets are useful since they correspond to (components of) the universal cover of embedded submanifolds. More specifically if $X\subset \O$ is $(\G,H)$-precisely invariant then $X/H$ embeds in $\O/\G$.

Lemma \ref{l:horoballapprox} tells us that for each peripheral subgroup we can find a horoball $\Hc_{int}$ that is contained in $\O_t$. The next lemma shows that, in addition, we can also arrange that these horoballs are precisely invariant with respect to the corresponding peripheral subgroup. 

\begin{lemma}\label{l:precinv}
The horoballs, $\Hc_{int}$, constructed in Lemma \ref{l:horoballapprox} can be chosen to be precisely invariant under the corresponding peripheral subgroup.
\end{lemma}

The following version of the Margulis lemma for properly convex domains will be crucial in the proof of Lemma \ref{l:precinv}.

\begin{lemma}[KMZ Lemma, \cite{CooperLongTillmann15,CM13}]\label{l:KMZ_lemma}
In every dimension $d$, there exists a positive constant $\varepsilon$ such that for every properly convex open set $\O$, for every $x \in \O$, for every discrete subgroup $\G$ of $\PGL(\O)$, the subgroup $\G_{\varepsilon}$ generated by the elements $\g \in \G$ such that $d_{\O}(x,\g \cdot x) < \varepsilon$ is virtually nilpotent.
 \end{lemma}

\begin{proof}[Proof of Lemma \ref{l:precinv}]
Let $\varepsilon$ be the constant of Lemma \ref{l:KMZ_lemma}. Let $\G_p$ be a peripheral subgroup of $\G$. Assume that $\G_p$ gives rise to a bent cusp, the case for a standard cusp can be treated similarly. Let $\Hc_{int}$ be the horoball guaranteed by Lemma \ref{l:horoballapprox}, and let $\Hc'_{int}$ be a smaller horoball with the same center. Since $\G_p$ is virtually a lattice in $B_d$ that arises from bending it contains a parabolic translation from $P^0_{d-1}$ which we call $\gamma$. We claim that every point on $\partial \Hc'_{int}$ is moved the same $\Hc_{int}$-Hilbert distance by $\gamma$. Let $x,y\in \partial \Hc'_{int}$. Since $B_d$ acts transitively on $\partial \Hc'_{int}$ we can find $\delta \in B_d$ such that $\delta y=x$. Therefore 
 $$d_{\Hc_{int}}(x,\gamma x)=d_{\Hc_{int}}(\delta y,\gamma\delta y)=d_{\Hc_{int}}(\delta y,\delta\gamma y)=d_{\Hc_{int}}(y,\gamma y),$$
 thus proving the claim.
 
Furthermore, since $\g$ is parabolic, if $z$ is a point on the boundary of an even smaller horoball then $d_{\Hc_{int}}(z,\gamma z)<d_{\Hc_{int}}(x,\gamma x)$ and this distance can be made arbitrarily close to zero by choosing the horoball to be sufficiently small. Thus, by shrinking $\Hc'_{int}$ if necessary, we can assume that $d_{\Hc_{int}}(z,\gamma z)<\varepsilon$ for $z\in \Hc'_{int}$. By the comparison property in Lemma \ref{p:compa_hilbert} we see that $d_{\O_t}(z,\gamma z)<\varepsilon$ for $z\in \Hc'_{int}$
 
We claim that $\Hc'_{int}$ is the desired precisely invariant horoball. By construction $\Hc'_{int}$ is $\G_p$-invariant. Next, suppose that $\tau\in \Gamma_t$ and that $\tau \Hc'_{int}\cap \Hc'_{int}\neq \varnothing$. Let $u\in \tau \Hc'_{int}\cap \Hc'_{int}$, hence $v = \tau^{-1} u \in \Hc'_{int}$. Observe that 

$$
 d_{\Omega_t}(u,\tau \gamma \tau^{-1}u)=d_{\Omega_t}(\tau^{-1} u,\gamma \tau^{-1}u)=d_{\Omega_t}(v,\gamma v)<\varepsilon,
$$

\noindent and so we see that $\tau\gamma\tau^{-1}$ also moves $u$ a distance less than $\varepsilon$. This implies that the group $\langle \gamma,\tau\gamma\tau^{-1}\rangle$ is virtually nilpotent. As an abstract group, $\Gamma_t$ is the fundamental group of a finite volume hyperbolic manifold, hence hyperbolic relatively to its peripheral subgroup. This implies that $\tau$ and $\gamma$ have a common fixed point for their action on the ideal boundary of $\HH^d$, and thus $\tau\in \G_p$. 
\end{proof}

By combining the previous few results we get the following Corollary. 

\begin{corollary}\label{c:endclassification}
If $\mathcal{M}_t=(\O_t,\G_t)$ is a of properly convex projective structure resulting from bending $M$ along $\Sigma$ then each end of $\mathcal{M}_t$ is either a standard or bent cusp.
\end{corollary}

\begin{proof}
Each end of $\mathcal{M}_t$ gives rise to a conjugacy class of peripheral subgroups. From Lemma \ref{l:latticeliegroup} we know that every peripheral subgroup $\G_p\subset \G_t$ is virtually a lattice in either $P_d$ or $B_d$. Furthermore, from Lemma \ref{l:horoballapprox} we see that for a $\G_p$-invariant horoball $\Hc_p\subset \O_t$. Finally, Lemma \ref{l:precinv} ensures that we can choose the $\Hc_p$ to be $(\G_t,\G_p)$-precisely invariant. As a result  we can find an invariant horoball in $\O_t$, and as a result the end corresponding to the conjugacy class of $\G_p$ is projectively equivalent to $\Hc_p/\G_p$. 
\end{proof}

\subsection{Volume of manifolds arising from bending}
We close this section by proving that the manifolds resulting from bending are always finite volume.  

\begin{theorem}\label{t:bendingfinitevolume}
 Let $M$ be a finite volume hyperbolic manifold and let $\Sigma$ be a finite volume totally geodesic hypersurface. Let $(\Mc_t=\Omega_t/\Gamma_t)_{t\in \R}$ be the projective manifolds obtained by bending $M$ along $\Sigma$, then $\Mc_t$ is a finite volume properly convex projective manifold. 
\end{theorem}

\begin{proof}
 Only the finite volume assertion remains to be proven. Since $M$ is topologically tame it has finitely many ends. Hence the set of peripheral subgroups of $\G$ is finite up to conjugation by $\G$. In order to simplify the exposition we assume that $M$ has a single cusp. We first deal with the case where the cusp is bent. Let $\G_p$ be a peripheral subgroup of $\G$. By Lemmas \ref{l:precinv} and \ref{l:horoballapprox} we can find a horoball $\Hc_{int}$ that is $(\G_t,\G_p)$-precisely invariant under $\G_p$ and centered at the peripheral face of $\G_p$. 
 
Thus we see that $\Hc_{int}/\G_p$ is an embedded submanifold of $\Omega_t/\Gamma_t$ the closure of whose complement is compact. Thus the proof will be complete if we can show that $\Hc_{int}/\G_p$ has finite Busemann volume. Since $\G_p$ is virtually a lattice in $B_d$, we can find a finite index subgroup $\G'_p$ which is a lattice in $B_d$. Furthermore, $\Hc_{int}/\G'_p$ is a finite sheeted cover of $\Hc_{int}/\G_p$, and so without loss of generality we can assume that $\G_{p}= \G'_p$ and the proof will thus be complete if we can show that $\Hc_{int}/\G_p$ is a finite volume submanifold of $\O_t/\Gamma_t$. 
 
Let $\Hc'$ be a slightly larger precisely invariant horoball with the same center as $\Hc_{int}$ such that $\Hc_{int} \subset \Hc'\subset \Omega_t$. By Theorem \ref{t:fvcuspnbhd} we see that $\Hc_{int}/\G_p$ is a finite volume submanifold of $\Hc'/\G_p$. By comparison properties \ref{p:compa_busemann} of the Busemann volume this implies that $\Hc_{int}/\G_p$ is a finite volume submanifold of $\O_t/\Gamma_t$. 
 
In the case of a standard cusp, the argument is similar. We conclude by remarking that the ellipsoid is the projective model of the hyperbolic space and well-known estimates of volume in hyperbolic space gives the finiteness of the volume of a standard cusp.
\end{proof}

\section{Geometry of ends in terms of homology}\label{s:homology}

This section discusses the relationship between the topology of the pair $(M,\Sigma)$ and the geometry of the ends of $M$ after bending along $\Sigma$. The fact that the geometry of the cusps is determined completely by topological information is somewhat surprising in light of the previous observation that nature of the structure on the cusp depends on a projective structure on $S^1$ associated to each end (see Remark \ref{r:S1affstruct}). 

Let $T$ be a cusp cross section of $M$. Since $T$ is a flat $(d-1)$-manifold, it is finitely covered by a $(d-1)$-torus, $T^\ast$. Let $(T\cap \Sigma)^\ast$ be the complete preimage of $T\cap \Sigma$ in $T^\ast$ under the aforementioned covering. Concretely, $(T\cap \Sigma)^\ast$ is a union of parallel $(d-2)$-tori in $T^\ast$. The covering map provides each component with an orientation and as a result we get a homology class $[(T\cap \Sigma)^\ast]\in H_{d-2}(T^\ast;\ZZ)$. The following theorem shows that this homology class characterizes the type of the structure on the cusp corresponding to $T$.

\begin{theorem}\label{homologydeterminesbending}
 Let $M$ be a finite volume hyperbolic $d$-manifold and let $\Sigma$ be an embedded totally geodesic hypersurface, and let $(\Mc_t=\Omega_t/\Gamma_t)_{t\in \R}$ be the family of projective manifolds obtained by bending $M$ along $\Sigma$. If $T$ is a cusp cross section of one of the cusps of $M$ then for $t\neq 0$ the cusp corresponding to $T$ in $\Mc_t$ is a bent cusp if and only if the homology class $[(T\cap \Sigma)^\ast]\in H_{d-2}(T^\ast;\ZZ)$ is non-trivial. 
\end{theorem}

\begin{proof}
 Let $\rho_t$ be the holonomy representation for the projective structure resulting from bending $M$ along $\Sigma$ and let $\Omega_t=D_t(\HH^d)$ be the (properly convex) image of the developing map for the aforementioned structure. If $\Sigma\cap T=\varnothing$ then it is clear that $[(T\cap\Sigma)^\ast]=0$. We have previously seen that in this case that the projective structure on the cusp corresponding to $T$ remains standard. Thus we can assume that $\Sigma$ intersects $T$.
 
Let $\Gamma_\infty$ be a peripheral subgroup for $T$, let $\Delta_\infty$ be a peripheral subgroup for one of the (parallel) cusp cross sections of $\Sigma$ that intersect $T$, and let $\gamma\in \G^{Tr}_\infty$ be an element whose image generates $\G^{Tr}_\infty/\Delta^{Tr}_\infty$. By Lemma \ref{l:precinv} we can find for each $t$ a horoball $\mathcal{H}_t\subset \Omega_t$ that is $(\rho_t(\G),\rho_t(\G_\infty))$-precisely invariant. Let $H_0=D_t^{-1}(\Hc_t)\subset \HH^d$. It is easy to see that $H_0$ is $(\G,\G_\infty)$-precisely invariant and it is not hard to see that $H_0$ is at bounded distance from a $(\G,\G_\infty)$-precisely invariant horoball. The cusp of $M$ corresponding to $T$ is a bent cusp if and only if $\mathcal{
H}_t/\rho_t(\Gamma_\infty^{Tr})$ is a bent cusp and so we turn our attention to this simpler projective manifold.
 
 There is a unique foliation of $H_0$ by a pencil of hyperplanes on which the action of $\Delta_\infty^{Tr}$  \merde{preserves each leaf of the foliation}. We call this pencil $C^\ast$ and we see that $C^\ast_t=D_t(C^\ast)$ gives rise to a foliation of $\mathcal{H}_t$ on which the action of $\rho_t(\Delta_\infty^{Tr})$  \merde{preserves each leaf of the foliation}. 
 
As a result the developing map $D_t:H_0\to \mathcal{H}_t$ induces a map $\overline{D}_t:\RR\to \R$ corresponding to collapsing the hyperplanes in $C^\ast$ and $C_t^\ast$ to points. Concretely, $\overline{D}_t$ is the developing map for the affine structure on $S^1$ mentioned in Remark \ref{r:S1affstruct}. Each of these affine structures gives rise to a holonomy representation $$\overline{\rho}_t:\ZZ\cong \Gamma_\infty^{Tr}/\Delta_\infty^{Tr}\to \mathrm{Aff}(\RR).$$

 
 \merde{Let $\overline{\gamma}$ be a generator of $\Gamma_\infty^{Tr}/\Delta_\infty^{Tr}$ and let $\gamma$ be an element of $\Gamma_\infty^{Tr}$ that projects to $\overline{\gamma}$}. We can regard $\gamma$ as a curve in $T^\ast$ and by Poincar\'e duality we see that $[(T\cap \Sigma)^\ast]=0$ if and only if the algebraic intersection of $\gamma$ with $(T\cap \Sigma)^\ast$ is zero. Let $\{t_i\}_{i=1}^k$ be the set of components of $(T\cap \Sigma)^\ast$. When we project from $T^\ast$ to $S^1$ each $t_i$ projects to a signed point, $(p_i,\varep_i)$, where $p_i$ in $S^1$ and $\varep_i=\pm1$ according to the algebraic intersection of the corresponding component with $\gamma$.  Let $a$ be the number of signed points where $\varep_i=1$ and $b$ be the number of signed points where $\varep_i=-1$. It is easy to see that $[(T\cap \Sigma)^\ast]=0$ if and only if $a=b$.  
 
 We now turn our attention to the developing map $\overline{D}_t$. When $t=0$ the developing map has image $\RR$. By conjugating by an element of $\Aff(\RR)$ we can assume that $\overline{\rho}_0(\overline{\gamma})$ is the translation $x\mapsto x+1$.  Each signed point $(p_i,\varep_i)$ can be lifted to a unique signed point in the interval $[0,1]\subset \RR$, which by abuse of notation we also call $(p_i,\varep_i)$. By renumbering, if necessary, we can assume that $p_i<p_j$ whenever $i<j$.
 
 The developing map $\overline{D}_t$ is obtained by successively modifying $\overline{D}_0$ in the following way. Each $p_i$ divides $\RR$ into two halves and $\overline{D}_t$ is obtained post composing the right half by the element of $\mathrm{Aff}(\RR)$ that fixes $p_i$ and whose linear part is multiplication by $e^t$ (resp.\ $e^{-t}$) if $\varep_i=1$ (resp.\ $\varep_i=-1$). 
  
As we have seen, $\mathcal{H}_t/\rho_t(\Gamma_\infty^{Tr})$ is a bent cusp if and only if $\overline{\rho}_t(\overline{\gamma})$ is a hyperbolic element of $\Aff(\RR)$. This is equivalent to $\overline{\rho}_t(\overline{\gamma})$ being a similarity of $\RR$, rather than an isometry.

Let $\delta>0$ be such that $\delta<p_1$. Under our previous identification we see that the points $0$ and $\delta$ are mapped by $\overline{\gamma}$ to $1$ and $1+\delta$, respectively. By equivariance, we see that $\overline{\rho}_t(\overline{\gamma})$ must map $\overline{D}_t(0)$ to $\overline{D}_t(1)$ and $\overline{D}_t(\delta)$  to $\overline{D}_t(1+\delta)$. By construction, $0$ and $\delta$ are to the left of all the $p_i$; and $1$ and $1+\delta$ are to the right of all the $p_i$. As a result we see that the distance between $\overline{D}_t(0)$ and $\overline{D}_t(\delta)$ is $\delta$ and the distance between $\overline{D}_t(1)$ and $\overline{D}_t(1+\delta)$ is $e^{(a-b)t}\delta$. Thus we see that $\overline{\rho}_t(\overline{\gamma})$ is an isometry if and only if $a=b$.  
\end{proof}

Theorem \ref{homologydeterminesbending} has the following immediate corollary

\begin{corollary}\label{c:strictconvexbending}
Under the hypotheses of Theorem \ref{homologydeterminesbending}; if $\Sigma$ is separating then each cusp $M$ remains standard after bending along $\Sigma$. Consequently, the projective structures obtained by bending along $\Sigma$ are all strictly convex. 
\end{corollary}

\begin{proof}
 If $\Sigma$ is separating then $[\Sigma]\in H_{d-1}(M;\ZZ)$ is trivial and thus $[T\cap \Sigma]\in H_{d-2}(T;\ZZ)$ is trivial for any cusp cross section $T$. The proof is completed by observing that $[(T\cap \Sigma)^\ast]$ is just a multiple of the image of $[T\cap \Sigma]$ under the transfer homomorphism from $H_{d-2}(T;\QQ)$ to $H_{d-2}(T^\ast;\QQ)$. Strict convexity of the resulting structures follows from \cite[Thm 11.6]{CooperLongTillmann15}.
\end{proof}

\section{Examples}\label{s:examples}

In this section we discuss examples of properly convex manifolds that arise from bending. The main results of this section are Theorem \ref{examples_strict_convex} and Theorem \ref{examples}, which show that there are examples of both strictly convex and properly, but not strictly convex finite volume manifolds in every dimension above 2.

\subsection{A 3-manifold with both standard and bent cusps}

We begin by describing a concrete 3-dimensional example. Let $M$ be the complement in $S^3$ of the Whitehead link. This manifold has two cusp cross sections $T_1$ and $T_2$ given by taking regular neighborhoods of the components $C_1$ and $C_2$ of the link (see Figure \ref{whitehead}). The manifold $M$ also contains a totally geodesic pair of pants $S$. This surface intersects $T_1$ in a single curve and so $[S\cap T_1]$ is a non trivial homology class in $H_1(T_1;\ZZ)$. By Theorem \ref{homologydeterminesbending} we see that when we bend $M$ along $S$ the cusp corresponding to $T_1$ becomes a bent cusp. 

On the other hand, $S$ intersects $T_2$ in two parallel, but oppositely oriented curves in $T_2$ and so we see that $[S\cap T_2]$ is a trivial class in $H_1(T_2;\ZZ)$ and so Theorem \ref{homologydeterminesbending} tells us that bending $M$ along $S$ results in the cusp corresponding to $T_2$ to remain standard. 

\begin{center}
 \begin{figure}[h]
  \includegraphics[scale=.4]{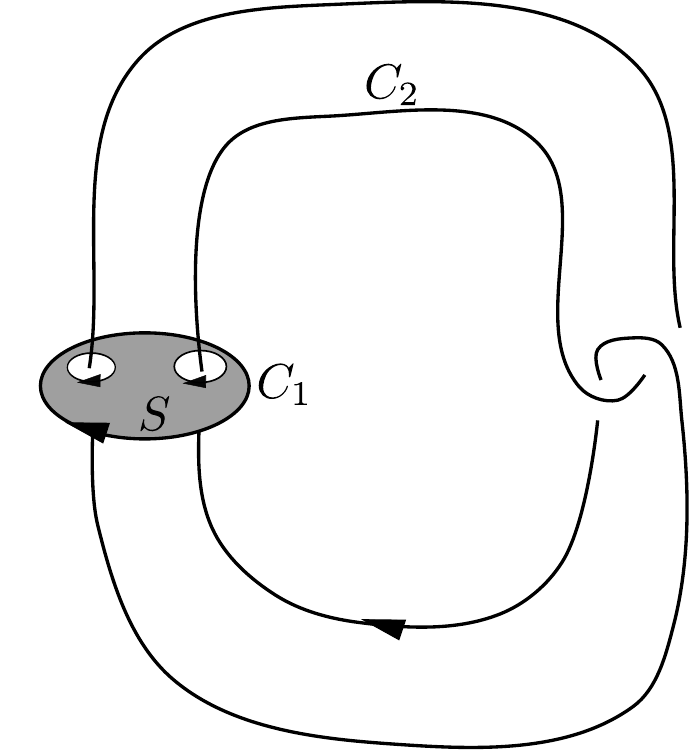}
  \caption{\label{whitehead}The Whitehead link contains a totally geodesic pair of pants }
 \end{figure}

\end{center}

\subsection{Non-strictly convex Examples}

Next, we show that for each dimension $d\geqslant 3$ there are properly convex manifolds with bent cusps. The precise statement of the result is:

\begin{theorem}\label{examples}
For each $d\geqslant 3$ there exists a properly convex $d$-manifold $M$ such that $M$ has finite volume and contains an end which is a bent cusp.  
\end{theorem}

\begin{proof}
 Let $\hat\Gamma=\PSO(Q_d)\cap \PSL_{d+1}(\ZZ)$ and $\hat\Lambda=\PSO(Q_d;d-1,1)\cap \hat \Gamma$. It is well known (see \cite{BorelHarishChandra}) that $\hat M:=\HH^d/\hat\Gamma$ is a non-compact finite volume orbifold that contains a totally geodesic immersion of the non-compact finite volume orbifold $\hat \Sigma:=\HH^{d-1}/\hat\Lambda$. By combining work of \cite{millson,Bergeron00} and \cite{collisions} we can find finite index subgroups $\Gamma\leqslant \hat \Gamma$ and $\Lambda\leqslant \hat\Lambda$ such that $M=\HH^d/\Gamma$ is a manifold containing a totally geodesic non-separating embedding of $\Sigma=\HH^{d-1}/\Lambda$ and whose cusp cross sections are all $(d-1)$-dimensional tori. 
 
Since $\Sigma$ is non-compact there is a cusp cross section, $T$, of $M$ that has non trivial intersection with $T$. Since $\Sigma$ is embedded we see that $\Sigma\cap T=\sqcup_{i=1}^k t_i$, where the $t_i$ are parallel $(d-2)$-dimensional tori embedded in $T$. 
 
Suppose that $k=1$. Then bending $M$ along $\Sigma$ will result in the cusp corresponding to $T$ becoming a bent cusp. If $k>1$ and $\Sigma$ is non-separating then it is possible that bending $T$ along the various component of $\Sigma\cap T$ will result in cancellation, in which case the cusp will remain standard. However, we claim that by passing to a finite sheeted cover of $M$ we can always arrange that $k=1$. This can be seen as follows: 
 
 Let $\tau_i$ be the fundamental group of $t_i$. Since each $t_i$ is contained in $T$ we see that the $\tau_i$ are all conjugate subgroups of $\Gamma$. However, since $\Sigma$ is an \merde{embedded} totally geodesic submanifold each subgroup $\tau_i$ corresponds to a distinct cusp cross section of $\Sigma$ and so the subgroups $\tau_i$ are pairwise non-conjugate subgroups of $\Lambda$.
 
 By construction, $M$ is an arithmetic manifold and so $\Gamma$ virtually retracts onto $\Lambda$ (see Theorem 1.4 and the comments at the end of \S9 in \cite{BergHagWise} for details). That is to say there is a finite index subgroup $\Gamma'$ of $\Gamma$ that contains $\Lambda$ and a homomorphism $r:\Gamma'\to \Lambda$ that restricts to the identity on $\Lambda$. Since $\Lambda\leqslant \Gamma'$ the embedding of $\Sigma$ into $M$ lifts to an embedding into $M'=\HH^d/\Gamma'$. The proof will be complete if we can show that the $\tau_i$ are pairwise non-conjugate in $\Gamma'$. This is done in \cite{collisions}, but the proof is short and so we include it for the sake of completeness. Suppose for contradiction that two of these subgroups, say $\tau_1$ and $\tau_2$, are conjugate in $\Gamma'$. Without loss of generality we can assume that there exists $\gamma\in \Gamma'$ such that $\gamma\tau_1\gamma^{-1}=\tau_2$. Since $\tau_1$ and $\tau_2$ are both subgroups of $\Lambda$ we see that
 $$\tau_2=r(\tau_2)=r(\gamma\tau_1\gamma^{-1})=r(\gamma)r(\tau_1)r(\gamma)^{-1}=r(\gamma)\tau_1r(\gamma)^{-1}.$$
Thus the groups $\tau_1$ and $\tau_2$ are conjugate in $\Lambda$, which is a contradiction. 
 \end{proof}

An immediate corollary of Theorem \ref{examples} is the following, which provides a partial answer to Question 3 in \cite{Marquis13}
\begin{corollary}
 In each dimension $d\geqslant 3$ there exist properly, but not strictly-convex manifolds with finite volume. 
\end{corollary}

\subsection{Strictly convex examples} In this subsection we show how to construct examples for which bending gives rise to strictly convex projective structures. 

\begin{theorem}\label{examples_strict_convex}
For each $d\geqslant 3$ there exists a strictly convex $d$-manifold $M=\Quo$ such that $M$ has finite volume and $\O$ is strictly convex.
\end{theorem}

\begin{proof}
Our ultimate goal is to produce a finite volume hyperbolic $d$-manifold that contains an embedded \emph{separating} totally geodesic hypersurface. This can be done as follows. Let $\hat\Gamma=\PSO(Q_d)\cap \PSL_{d+1}(\ZZ)$. There is an obvious embedding of the group $\PO(Q_{d-1})$ (i.e. the full isometry group hyperbolic $(d-1)$-space) into the stabilizer of $\HH^{d-1}_0$ in $\PSO(Q_d)$. Let $\PO(Q_d;d-1,1)$ denote its image and let $\hat \Lambda=\PO(Q_d; d-1,1)\cap \hat \Gamma$.  It is easy to see that the orientable hyperbolic $d$-orbifold $\hat M:=\HH^{d}/\hat \Gamma$  contains an immersed totally geodesic copy of the non-orientalbe hyperbolic $(d-1)$-orbifold $\hat \Sigma:=\HH^{d-1}/\hat \Lambda$.  

By work of Long--Reid \cite[\S 3]{LongReid01} it is possible find finite sheeted covers $M$ of $\hat M$ and $\Sigma$ of $\hat \Sigma$ as well as a totally geodesic embedding $\Sigma\hookrightarrow M$ whose image is separating. Technically, the results in \cite{LongReid01} require $\hat M$ and $\hat \Sigma$ to be closed, however a close examination of their proof reveals that the same argument works in the case where $\hat M$ and $\hat \Sigma$ are finite volume. The result then follows by applying Corollary \ref{c:strictconvexbending} and Theorem \ref{t:bendingfinitevolume}.
\end{proof}

\subsection{Proof of Theorem \ref{maintheorem}}

We close this section by proving Theorem \ref{maintheorem}. In order to do this we need a few preliminary results.

\begin{lemma}\label{l:irreducible}
Let $\O_t$ be a properly convex domain obtained by bending $M$ along $\Sigma$. Then $\O_t$ is irreducible.
\end{lemma}

\begin{proof}
Let $\O$ be one of the domains constructed using Theorem \ref{examples} or Theorem \ref{examples_strict_convex}, using a totally geodesic hypersurface $\Sigma$. By construction, those groups contains the fundamental $\pi_1(M_{\Sigma})$ of one of the connected component of $M\smallsetminus \Sigma$, but the group $\pi_1(M_{\Sigma})$ is changed during the bending by a conjugation, and the group $\pi_1(M_{\Sigma})$ acts strongly irreducibly on $\R^{d+1}$ at time $t=0$, since its limit set is not included in an hyperplane of $\partial \Hb^d$, so it acts strongly irreducibly at any time. Hence, $\G$ acts strongly irreducibly on $\R^{d+1}$. Thus $\O$ is irreducible, since any decomposition of $\O$ as a non-trivial product would imply the existence of a finite index non-irreducible subgroup of $\G$.
\end{proof}

We now turn our attention to the proof of Theorem \ref{zariskidense} which we will need in order to prove that the domains $\O$ constructed by bending are non-homogeneous. To complete the proof we use a Theorem of Benoist \cite{AutConvBenoist}. In fact, we need a small improvement, given by Lemma 7.14 of  \cite{cox_gr_hil_geom}.

\begin{lemma}\label{detail_adh_zar}
Let $\G$ be a strongly irreducible subgroup of $\mathrm{PGL}_{d+1}(\R)$ preserving a properly convex open set. Let $G$ be the connected component of the Zariski closure of $\G$. Suppose there exists a point $x$ in the limit set, $\Lambda_G$, of $G$ and a Zariski closed subgroup $H$ of $G$ such that the orbit $H \cdot x$ is a sub-manifold of $\PP^d$ of dimension at least $d-1$. Then $G$ is conjugate to $\mathrm{PSO}_{d,1}(\R)$ or $G=\mathrm{PGL}_{d+1}(\R)$.
\end{lemma}

\begin{proof}[Proof of Theorem \ref{zariskidense}]
In order to apply Lemma \ref{detail_adh_zar}, we just need to set $H$ to be the Zariski closure of one of the peripheral subgroups of $\G$. The group $P_d$ is Zariski closed, and so if the cusp is standard then by Theorem \ref{t:holo_end_standorbend} we can assume (after conjugating) that $H=P_d$. However, the group $B_d$ is not Zariski closed (since it contains entries with the transcendental function $e^t$). Therefore, using a similar argument we find that $H$ is $d$ dimensional and consists of matrices of the form 
$$
\begin{pmatrix}
1 & 0 & v^t & u\\
0 & w & 0 & 0\\
0 & 0 & I & v\\
0 & 0 & 0 & 1 
\end{pmatrix}
$$

where $v\in \R^{d-2}$, $I$ is the $(d-2)\times (d-2)$ identity matrix, and $u,w\in \R$.   Thanks to the analysis  of Section \ref{s:generalizedcusps}, we see that generically, the orbits of $H$ contain horospheres of the type introduced in \ref{s:generalizedcusps}, and hence these orbits is at least dimension $d-1$. Moreover, for any $x \in \R\Pb^d$ in the complement of a particular hyperplane the $H$-orbit of $x$ contains a horosphere. So, one can find a point of $\Lambda_G$ whose $H$-orbit is at least of dimension $d-1$.

Thus, Lemma \ref{detail_adh_zar} shows that the Zariski-closure of $\G$ is either $\mathrm{PSO}_{d,1}(\R)$ or $G=\mathrm{PGL}_{d+1}(\R)$. If $t \neq 0$ then it cannot be $\mathrm{PSO}_{d,1}(\R)$, since the matrix $c_t$ (introduced in equation \ref{e:c_t}) does not normalize  $\mathrm{PSO}_{d,1}(\R)$, for $t \neq 0$.

\end{proof}

Finally, we prove that $\O$ is not homogeneous.

\begin{lemma}\label{l:nonhomogeneous}
For $t\neq 0$ the domains $\O_t$ constructed by bending $M$ along $\Sigma$ are non-homogeneous
\end{lemma} 

\begin{proof}
In order to prove the result, we show that $\mathrm{PGL}(\O_t)$ is a discrete subgroup of $\mathrm{PGL}_{d+1}(\R)$. The group $\G_t$ is Zariski-dense, so the group $\mathrm{PGL}(\O_t)$ is also Zariski-dense.

First, we stress that a Zariski-dense subgroup, $\Lambda$, of an almost simple Lie group, i.e. a Lie group with a simple Lie algebra, is either discrete or dense, since the closure of $\Lambda$ for the usual topology is normalized by $\Lambda$, and so normalized by its Zariski-closure.

Now, the group $\mathrm{PGL}(\O_t)$ is not dense in $\mathrm{PGL}_{d+1}(\R)$ since it preserves the convex $\O_t$. Hence, the group $\mathrm{PGL}(\O_t)$ is discrete. 
\end{proof}

\begin{remark}
One consequence of the proof of Lemma \ref{l:nonhomogeneous} is that the index of $\G_t$ in $\mathrm{PGL}(\O_t)$ is finite since the quotient of $\O_t$ by both groups is of finite volume.
\end{remark}

\begin{proof}[Proof of Theorem \ref{maintheorem}]
By Theorems \ref{examples} and \ref{examples_strict_convex} we can find examples of strictly convex and non-strictly convex properly convex $\O$ via bending. By Theorem \ref{t:bendingfinitevolume} we see that these $\O$ are quasi-divisible.

By Lemma \ref{l:irreducible} we see that these $\O$ are always irreducible. Finally, by Lemma \ref{l:nonhomogeneous} we see that these $\O$ are always non-homogenous. 
\end{proof}

\bibliographystyle{alpha}

\begin{thebibliography}{BHW11}

\bibitem[Bal13]{BallasThesis}
S.~Ballas.
\newblock {\em Flexibility and Rigidity of Three-Dimensional Convex Projective
  Structures}.
\newblock {Ph.D. thesis}, University of Texas, 2013.

\bibitem[Bal14]{deform_sam}
S.~Ballas.
\newblock Deformations of noncompact projective manifolds.
\newblock {\em Algebr. Geom. Topol.}, 14(5):2595--2625, 2014.

\bibitem[Bal15]{8knot_sam}
S.~Ballas.
\newblock Finite volume properly convex deformations of the figure-eight knot.
\newblock {\em Geom. Dedicata}, 178:49--73, 2015.

\bibitem[{Bal}18]{sam_gen_cusps}
Samuel~A {Ballas}.
\newblock {Constructing convex projective 3-manifolds with generalized cusps}.
\newblock {\em arXiv e-prints}, page arXiv:1805.09274, May 2018.

\bibitem[BCL17]{BalCoopLeit}
Samuel~A. {Ballas}, Daryl {Cooper}, and Arielle {Leitner}.
\newblock {Generalized Cusps in Real Projective Manifolds: Classification}.
\newblock {\em arXiv e-prints}, page arXiv:1710.03132, Oct 2017.

\bibitem[BDL18]{ballas_danciger_lee}
Samuel~A. Ballas, Jeffrey Danciger, and Gye-Seon Lee.
\newblock Convex projective structures on nonhyperbolic three-manifolds.
\newblock {\em Geom. Topol.}, 22(3):1593--1646, 2018.

\bibitem[Ben00]{AutConvBenoist}
Y.~Benoist.
\newblock {Automorphismes des c{\^o}nes convexes}.
\newblock {\em Invent. Math.}, 141(1):149--193, 2000.

\bibitem[Ben04]{CD1}
Y.~Benoist.
\newblock {Convexes divisibles. {I}}.
\newblock In {\em {Algebraic groups and arithmetic}}, pages 339--374. Tata
  Inst. Fund. Res., Mumbai, 2004.

\bibitem[Ben06]{cd4}
Y.~Benoist.
\newblock {Convexes divisibles. {IV}. {S}tructure du bord en dimension 3}.
\newblock {\em Invent. Math.}, 164(2):249--278, 2006.

\bibitem[Ber00]{Bergeron00}
N.~Bergeron.
\newblock Premier nombre de {B}etti et spectre du laplacien de certaines
  vari\'et\'es hyperboliques.
\newblock {\em Enseign. Math. (2)}, 46(1-2):109--137, 2000.

\bibitem[BHC62]{BorelHarishChandra}
A.~Borel and M.~Harish-Chandra.
\newblock Arithmetic subgroups of algebraic groups.
\newblock {\em Ann. of Math. (2)}, 75:485--535, 1962.

\bibitem[BHW11]{BergHagWise}
N.~Bergeron, F.~Haglund, and D.~Wise.
\newblock Hyperplane sections in arithmetic hyperbolic manifolds.
\newblock {\em J. Lond. Math. Soc. (2)}, 83(2):431--448, 2011.

\bibitem[BL18]{BallasLongThin}
Samuel {Ballas} and D.~D. {Long}.
\newblock {Constructing thin subgroups of SL(n+1,R) via bending}.
\newblock {\em To appear in Algebr. Geom. Topol.}, page arXiv:1809.02689, Sep
  2018.

\bibitem[Bob19]{bobb}
Martin~D. Bobb.
\newblock Convex projective manifolds with a cusp of any non-diagonalizable
  type.
\newblock {\em J. Lond. Math. Soc. (2)}, 100(1):183--202, 2019.

\bibitem[CG05]{cgorb}
S.~Choi and W.~Goldman.
\newblock The deformation spaces of convex {$\Bbb{RP}^2$}-structures on
  2-orbifolds.
\newblock {\em Amer. J. Math.}, 127(5):1019--1102, 2005.

\bibitem[Cho94]{annuli_2}
S.~Choi.
\newblock {Convex decompositions of real projective surfaces. {II}.
  {A}dmissible decompositions}.
\newblock {\em J. Differential Geom.}, 40(2):239--283, 1994.

\bibitem[CL15]{CL15}
S.~Choi and G.-S. Lee.
\newblock {Projective deformations of weakly orderable hyperbolic {C}oxeter
  orbifolds}.
\newblock {\em Geom. Topol.}, 19(4):1777--1828, 2015.

\bibitem[CLM18]{CLM_survey}
S.~{Choi}, G.-S. {Lee}, and L.~{Marquis}.
\newblock Deformations of convex real projective structures on manifolds and
  orbifolds.
\newblock In {\em Handbook of Group Actions (Vol. III)}, volume~40 of {\em
  ALM}, chapter~10. Edited by Lizhen Ji, Athanase Papadopoulos and Shing-Tung
  Yau., 2018.

\bibitem[CLT06]{CLT_c}
D.~Cooper, D.~Long, and M.~Thistlethwaite.
\newblock {Computing varieties of representations of hyperbolic 3-manifolds
  into {${\rm SL}(4,\Bbb R)$}}.
\newblock {\em Experiment. Math.}, 15(3):291--305, 2006.

\bibitem[CLT07]{CLT_flexing}
D.~Cooper, D.~Long, and M.~Thistlethwaite.
\newblock {Flexing closed hyperbolic manifolds}.
\newblock {\em Geom. Topol.}, 11:2413--2440, 2007.

\bibitem[CLT15]{CooperLongTillmann15}
D.~Cooper, D.~Long, and S.~Tillmann.
\newblock On convex projective manifolds and cusps.
\newblock {\em Adv. Math.}, 277:181--251, 2015.

\bibitem[CLT18]{CLT15}
D.~{Cooper}, D.~{Long}, and S.~{Tillmann}.
\newblock {Deforming convex projective manifolds}.
\newblock {\em Geom. Topol.}, 22:1349--1404, 2018.

\bibitem[CM13]{CM13}
M.~Crampon and L.~Marquis.
\newblock {Un lemme de {K}azhdan-{M}argulis-{Z}assenhaus pour les
  g{\'e}om{\'e}tries de {H}ilbert}.
\newblock {\em Ann. Math. Blaise Pascal}, 20(2):363--376, 2013.

\bibitem[CVV04]{aire_cvv}
B.~Colbois, C.~Vernicos, and P.~Verovic.
\newblock L'aire des triangles id\'eaux en g\'eom\'etrie de {H}ilbert.
\newblock {\em Enseign. Math. (2)}, 50(3-4):203--237, 2004.

\bibitem[FG07]{baby_fock}
V.~Fock and A.~Goncharov.
\newblock {Moduli spaces of convex projective structures on surfaces}.
\newblock {\em Adv. Math.}, 208(1):249--273, 2007.

\bibitem[Gol77]{Goldmanthesis}
W.~Goldman.
\newblock Affine manifolds and projective geometry on surfaces.
\newblock BA thesis, Princeton University, New Jersey, 1977.

\bibitem[Gol90]{Gconv}
W.~Goldman.
\newblock Convex real projective structures on compact surfaces.
\newblock {\em J. Differential Geom.}, 31(3):791--845, 1990.

\bibitem[JM87]{JoMi}
D.~Johnson and J.~Millson.
\newblock Deformation spaces associated to compact hyperbolic manifolds.
\newblock In {\em Discrete groups in geometry and analysis ({N}ew {H}aven,
  {C}onn., 1984)}, volume~67 of {\em Progr. Math.}, pages 48--106. Birkh\"auser
  Boston, Boston, MA, 1987.

\bibitem[Kos68]{Ko}
J.-L. Koszul.
\newblock D\'eformations de connexions localement plates.
\newblock {\em Ann. Inst. Fourier (Grenoble)}, 18(fasc. 1):103--114, 1968.

\bibitem[LR01]{LongReid01}
D.~Long and A.~Reid.
\newblock Constructing hyperbolic manifolds which bound geometrically.
\newblock {\em Math. Res. Lett.}, 8(4):443--455, 2001.

\bibitem[Mar10a]{ecima}
L.~Marquis.
\newblock {Espace des modules de certains poly{\`e}dres projectifs miroirs}.
\newblock {\em Geom. Dedicata}, 147:47--86, 2010.

\bibitem[Mar10b]{marquis_moduli_surf}
L.~Marquis.
\newblock {Espace des modules marqu{\'e}s des surfaces projectives convexes de
  volume fini}.
\newblock {\em Geom. Topol.}, 14(4):2103--2149, 2010.

\bibitem[Mar12a]{Mar}
L.~Marquis.
\newblock Exemples de vari\'et\'es projectives strictement convexes de volume
  fini en dimension quelconque.
\newblock {\em Enseign. Math. (2)}, 58(1-2):3--47, 2012.

\bibitem[Mar12b]{Surf_ludo}
L.~Marquis.
\newblock {Surface projective convexe de volume fini}.
\newblock {\em Ann. Inst. Fourier {\rm (}Grenoble{\rm )}}, 62(1):325--392,
  2012.

\bibitem[Mar14]{Marquis13}
L.~Marquis.
\newblock Around groups in {H}ilbert geometry.
\newblock In {\em Handbook of {H}ilbert geometry}, volume~22 of {\em IRMA Lect.
  Math. Theor. Phys.}, pages 207--261. Eur. Math. Soc., Z\"urich, 2014.

\bibitem[Mar17]{cox_gr_hil_geom}
Ludovic Marquis.
\newblock Coxeter group in {H}ilbert geometry.
\newblock {\em Groups Geom. Dyn.}, 11(3):819--877, 2017.

\bibitem[Mil76]{millson}
John~J. Millson.
\newblock On the first {B}etti number of a constant negatively curved manifold.
\newblock {\em Ann. of Math. (2)}, 104(2):235--247, 1976.

\bibitem[MRS13]{collisions}
D.~McReynolds, A.~Reid, and M.~Stover.
\newblock Collisions at infinity in hyperbolic manifolds.
\newblock {\em Math. Proc. Cambridge Philos. Soc.}, 155(3):459--463, 2013.

\bibitem[Poi82]{poincare}
H.~Poincar{\'e}.
\newblock {Th{\'e}orie des groupes fuchsiens}.
\newblock {\em Acta Math.}, 1(1):1--76, 1882.

\bibitem[Sar14]{sarnak_survey}
P.~Sarnak.
\newblock Notes on thin matrix groups.
\newblock In {\em Thin groups and superstrong approximation}, volume~61 of {\em
  Math. Sci. Res. Inst. Publ.}, pages 343--362. Cambridge Univ. Press,
  Cambridge, 2014.

\end{thebibliography}

\end{document}